\documentclass{amsart}

\usepackage{amssymb}
\usepackage{amsmath,amsfonts,amsthm,amstext}
\usepackage{amsmath}
\usepackage{mathrsfs}  
\usepackage[all]{xy}
\usepackage{xcolor}
\usepackage{tikz-cd}

\newtheorem{theorem}{Theorem}[section]
\newtheorem{lemma}[theorem]{Lemma}

\newtheorem{prop}[theorem]{Proposition}
\newtheorem{cor}[theorem]{Corollary}

\newcommand{\Z}{\mathbb{Z}}
\newcommand{\z}{\mathbb{Z}}

\newcommand{\X}{{\mathcal X}}
\theoremstyle{definition}

\begin{document}

		\title[]{
		On the residual finiteness of the non-abelian exterior square of wreath products and the Grigorchuk group} \author{Ayrton  Anjos Teixeira, Dessislava H. Kochloukova}
	\address{State University of Campinas (UNICAMP), SP, Brazil \\
		\newline
		email : desi@ime.unicamp.br
		} 
	\email{}
	\date{}

\begin{abstract}  We show some sufficient conditions ( in terms of properties of $A$, $B$ and $X$) for the non-abelian exterior square $G \wedge G$ to be residually finite, where $G $ is the wreath product $A \wr B$ or the combinatorial wreath product  $G = A \wr_X B$, in the latter case $B$ is a free group. 
We prove that for the Grigorchuk group $G$  the non-abelian exterior square $G \wedge G$  and the non-abelian tensor product $G \otimes G$ are residually finite. We show conditions that imply that $A \wr B$ is not (cohomologically) good  in dimension $\leq 2$. 
\end{abstract}

\maketitle

\section{Introduction}

 In this paper we study some sufficient conditions that force the non-abelian exterior square $G \wedge G$   to be residually finite where $G$ is a specific  wreath product or the Grigorchuk group. In all cases the original group $G$ is residually finite.   
 The non-abelian tensor product $G \otimes G$ and its quotient $G \wedge G$ are very well studied, starting from the works of
   Dennis,  Brown and Loday  in \cite{D}, \cite{B-L1,B-L2}. A recurrent theme in the study of the non-abelian tensor product is the fact that many classes of groups are closed to applying the non-abelian tensor product. For example 
by results of  Brown, Johnson, Robertson, Ellis and Moravec  if $G$ is in one of the following classes  of groups:  $\pi$-finite ($\pi$ is a set of primes), nilpotent, soluble, polycyclic, finite then  $G\otimes G$  is  in the same class \cite{B-J-R},  \cite{Ellis},  \cite{M}. 
Still a basic property as finite generation is not preserved by the non-abelian tensor product. But if $G$ is a finitely generated group, the group $G \otimes G$  though   not generally finitely generated, embeds in the finitely generated group  $\nu(G)$ defined by Rocco in \cite{Norai}.

The study of residual finiteness of the non-abelian tensor square $G \otimes G$ was started in \cite{AT-BR-K}. One of the main results of \cite{AT-BR-K} is that if $G$ is residually finite and the map $\theta: H_2(G, \mathbb{Z}) \to \varprojlim_U H_2(G/U, \Z) $ is injective (here $\widehat{G} = \varprojlim_U G/U$ with $G/U$ finite)   then the non-abelian tensor square $G \wedge G$ is residually finite  and if furthermore the abelianisation of $G$ does not contain 2-torsion and is finitely generated  then the non-abelian tensor square $G \otimes G$ is residually finite. 
The latter is a corollary of the result of Blyth, Fumagalli and Morigi that under the above restrictions on the abelianization of $G$  the map $G \otimes G \to G \wedge G$ splits.
The map  $\rho: H_2(G, \mathbb{Z}) \to H_2( \widehat{G}, \widehat{\Z})$ factors through 
$ \varprojlim_U H_2(G/U, \Z)$, hence the injectivity of $\rho$ implies the injectivity of $\theta$. 

For  $G = A \wr B$ a description of  $H^2(G, \mathbb{Z})$  was given by Reid in \cite{Re} and a description of $H_2(G, \mathbb{Z})$ was given by Blackburn in \cite{Bl}. We give a homological proof of Blackburn result in Theorem \ref{T1} and get some partial results on the Schur multiplier of the combinatorial wreath product in Theorem \ref{TT1} and Lemma \ref{LL1}.

The first class of groups we study in this paper is the combinatorial wreath product $G = A \wr_X B$ and the particular case of wreath product $A \wr B$.
The classification of the finitely presented combinatorial wreath product $G = A \wr_X B$ was established by Cornulier in \cite{C} and the classification of homological type $FP_m$ and of homotopical  type $F_m$ (under some mild restrictions)  was developed by Bartholdi, Cornulier, Kochloukova in \cite{B-C-K} and Kropholler and Martino in \cite{Kr-Ma}.

Gruenberg proved in \cite[Thm. 3.2]{Gru} that  $G =  A \wr B$ is residually finite precisely when  both $A$ and $B$ are residually finite and either $A$ is abelian or $B$ is finite.  Later this was generalised by Cornulier in \cite[Prop. 3.2]{C2} to the case of general combinatorial wreath product $G = A \wr_X B$. In this case, assuming $A$ is non-trivial,  $G$ is residually finite precisely when both $A$ and $B$ are residually finite and one of the following holds: either  all stabilisers of the action of $B$ on $X$ have finite index in $B$ (if $B$ acts with finitely many orbits on $X$, this is equivalent to $X$ is finite) or $A$ is abelian and every stabilizer $H$ of the action of $B$ on $X$ is closed in the profinite topology of $B$.

 For a $R$-module $V$ we write $V \square_{R}  V$ for $V \otimes_{R} V / \langle a \otimes b + b \otimes a \ | \ a,b \in V \rangle$. We will use this for $R = \Z$ i.e. $V$ is an abelian group.  For $W = \varprojlim_i W_i$ with each $W_i$ a finite abelian group, $W$ is a profinite abelian group  i.e a profinite $\widehat{\Z}$-module, we write $W \widehat{\otimes}_{\widehat{\Z}} W$, $W \widehat{\wedge}_{\widehat{\Z}} W$ and $W \widehat{\square}_{\widehat{\Z}} W$   for $ \varprojlim_i W_i {\otimes}_{\Z} W_i$,  $ \varprojlim_i W_i {\wedge}_{\Z} W_i$ and  $ \varprojlim_i W_i {\square}_{\Z} W_i$ respectively, where $W_i {\otimes}_{\Z} W_i$ and $ W_i {\wedge}_{\Z} W_i$ denote the standard constructions of a tensor and exterior  product over a ring i.e. this corresponds to the more general construction of tensor and exterior product applied  to abelian groups that act trivially on each other. It is easy to see that $W \widehat{\otimes}_{\widehat{\Z}} W$, $W \widehat{\wedge}_{\widehat{\Z}} W$ and $W \widehat{\square}_{\widehat{\Z}} W$ are dependent on $W$ but not on the choice of inverse system $(W_i)$.

\medskip
{\bf  Theorem A} 
 \label{important} 
{\it Suppose that $G = A \wr B$ is a residually finite group with  $A$  an abelian group  such that:

1) the map  $  A \wedge_{\Z} A  \to  \widehat{A} \widehat{\wedge}_{\widehat{\Z}} \widehat{A}   $, induced by the canonical map $A \to \widehat{A}$, is injective;

2)
 $A \otimes_{\Z} A \to \widehat{A} \widehat{\otimes}_{\widehat{\Z}} \widehat{A}$  is injective   if $B$ has an element $b$ of order   $|b| \in  \{ 3,4, \ldots \} \cup \{ \infty \}$  ;

3)  $A \square_{\Z} A \to \widehat{A} \widehat{\square}_{\widehat{\Z}} \widehat{A}$ is injective   if $B$ has an element $b$ of order   $|b| = 2$;

 4) the map  $H_2(B, \mathbb{Z}) \to \varprojlim_i H_2( B_i,{\Z})$ is injective where $\varprojlim_i B_i = \widehat{B}$ with $B_i$ finite.

Then the map $H_2(G, \mathbb{Z}) \to  \varprojlim_i H_2 (A_i \wr B_i, \Z)$ is injective, where $\widehat{G} =\varprojlim_i A_i \wr B_i$, $\widehat{A} =  \varprojlim_i A_i$, $\widehat{B} =  \varprojlim_i B_i$  with $A_i$ and $B_i$ finite groups.   In particular,  $G \wedge G$ is residually finite.

 {\bf Remark} Note that if $ \widehat{G} \simeq \varprojlim_U G/ U$ with $G/ U$ finite then  $\varprojlim_U H_2(G/ U, \Z)$ depends only on $ \widehat{G}$ i.e. if $\widehat{G} \simeq \varprojlim_V G/ V$ with $G/ V$ finite then   $\varprojlim_V H_2(G/ V, \Z) \simeq \varprojlim_{V, U} H_2(G/ U \cap V, \Z) \simeq \varprojlim_U H_2(G/ U, \Z) $. 
}
\medskip

In Theorem \ref{oldThmB} we prove a version of Theorem A in the case $A$ is not abelian but $B$ is finite.
Theorem A is proved in Chapter \ref{ChapterResFinWr}.  

  The non-abelian  $q$-exterior square $G \wedge^q G$ was defined   by Conduche and Rodr\'iguez-Fern\'andez  in  \cite{C-R} for an integer $q \geq 1$. 
In Theorem \ref{q-wreath-abelian} and Theorem \ref{q-section-finite} in section \ref{ChapterResFinWr} we get  conditions  for $G \wedge^q G$ to be residually finite for $G = A \wr B$  a residually finite group and $q \geq 2$.

 Theorem B is proved in Section \ref{ChapterFree}.  We use homological methods to show that  the map $H_2(G, \Z) \to H_2( \widehat{G}, \widehat{\Z})$ is injective and then apply 
the main result of \cite{AT-BR-K} stated in the introduction as  Theorem \ref{A-L-Da}.  To prove Theorem B, b) we use homological methods that reduce the problem to  Proposition \ref{iff25}. 

\medskip
{\bf Theorem B} {\it Let $G =  A \wr_X B$ be residually finite, where $B$ is a finitely generated free group, $A$ is a finitely generated  free  abelian group, $B$ acts on $X$ with finitely many orbits  and each stabiliser $stab_B(x)$ is finitely generated for $x \in X$.  Then 

a) the map $H_2(G, \Z) \to H_2( \widehat{G}, \widehat{\Z})$ is injective. In particular $G \wedge G$ is residually finite

b) for every integer $ q \geq 2$  the map $H_2(G, \Z/ q \Z) \to H_2( \widehat{G}, {\Z}/ q \Z)$   is  injective. In particular $G \wedge^q G$ is residually finite.}

\medskip 
The proof of Theorem B uses the description of the Schur multiplier of $G =  A \wr_X B$ given in Lemma \ref{discrete-homology}. In Section \ref{homCom} we
continue the study of the Schur multiplier of $G =  A \wr_X B$ and believe Theorem \ref{TT1} is of independent interest though it is no applicable to finite groups, hence do not help to understand whether $G \wedge G$ is residually finite.

Next we consider the Grigorchuk group $G$,   which is the first known group  of intermediate growth  \cite{Grig}. The Grigorchuk group is a self-similar group,in particular it is residually finite. Since every element in the Grigorchuk group has  order a power of 2, it is residually $2$-finite.   Theorem C  follows from Corollary \ref{GG1}, Proposition \ref{GG20} and Proposition \ref{GG30}. An important ingredient of the proof of Corollary \ref{GG1} is   Benli's  calculation from \cite{Benli}  of $H_2( G_n, \Z)$ for  some finite quotients $G_n$ of $G$ that determine the profinite topology on $G$.

\medskip
{\bf Theorem C} {\it Let  $G$ be the Grigorchuk group. Then

a)  the map $H_2(G, \Z) \to \varprojlim_n H_2(G_n, {\Z})$  is injective, where  the profinite completion  $ \widehat{G} \simeq \varprojlim_n G_n$  is a pro-$2$ group with $G_n$ 2-finite quotients of $G$.
In particular  the non-abelian exterior square $G \wedge G$ is residually 2-finite.

 b) Nevertheless $G^{ab}$ is a 2-group, $G \otimes G \simeq \Delta(G) \times (G \wedge G)$. In particular  $G \otimes G$ is residually  2-finite.
 
  c) For every integer $q \geq 2$ the map  $H_2(G, \Z/ q \Z) \to H_2(\widehat{G}, {\Z}/ q \Z)$ is injective.
In particular  the non-abelian $q$-exterior square $G \wedge^q G$ is residually finite.}

\medskip
Finally we consider the class of groups $\mathcal{A}_2$,  i.e. (cohomologically) good in dimension $\leq 2$, defined by Serre in \cite{Serre}. The importance of the class $\mathcal{A}_2$ lies in the fact that every extension of a  finitely generated residually finite group by a residually finite group from the class $\mathcal{A}_2$ is again residually finite.

In \cite[Prop. 1, Section 2]{Er} Erschler showed that there exist (infinitely) many finite-by-Grigorchuk groups that are not residually finite, in particular the Grigorchuk group $G$  is not in $\mathcal{A}_2$. 
In \cite[Prop. 2, Section 3]{Er} Erschler  constructs a  non-residually finite group denoted by $G_{Hall}^2$, we denote by $\Gamma$, such that its center $C(\Gamma) \simeq (\Z/ 2 \Z)^{\infty}$ and $\Gamma/ C(\Gamma) \simeq \Z \wr \Z$. This easily implies that $\Z \wr \Z \not\in {\mathcal{A}}_2$. Indeed if $\Z \wr \Z \in {\mathcal{A}}_2$ then for every  subgroup $G_i$   of finite index in $C(\Gamma)$ we have that $\Gamma/ G_i$ is residually finite, hence there is a subgroup $\Gamma_i$ of finite index in $\Gamma$  such that $\Gamma_i \cap C(\Gamma) \subseteq G_i$. By choosing $G_i$ such that $\cap_i G_i = 1$ we get a contradiction with $\Gamma$ is not residually finite.

 In \cite{BR-K} Barroso Rocha and Kochloukova proved that the cohomological and homological version of goodness coincide. 
In the proof of the following theorem we show that many wreath products are not homologically good in dimension 2, hence are not in the class ${\mathcal{A}}_2$.

\medskip
{\bf Theorem D} {\it  Suppose that $B$ is an infinite,  residually finite group, $G = A \wr B$ is countable, residually finite  and  $A$  is  non-trivial.   Then  one  (or both) of the following conditions holds:

a) $A \wr B \not\in {\mathcal{A}}_2$,

b)  $dim_{\mathbb{F}_p}  A  \otimes_{\Z} \mathbb{F}_p \leq 1$ for every prime $p> 2$, $A = 2A$ and $B$ has only finitely many elements of  order different from two.
}

\medskip
{\bf Acknowledgements}   Anjos Teixeira is supported by a PhD grant process No. 168757/2023-0, Conselho Nacional de Desenvolvimento Científico e Tecnológico - CNPq  and Kochloukova is partially supported by bolsa de produtividade em pesquisa CNPq 301222/2026-6 and FAPESP 2024/14914-9. 

\section{Preliminaries}

 \subsection{Structure of $G \otimes G$}

We quote some  general results on  the non-abelian tensor square $G \otimes G$ and the non-abelian exterior square $G \wedge G$.

In \cite{Norai}  Rocco defined the group $\nu(G)$ by the following presentation
$$\nu(G) = \langle G, {G}^{\varphi} \ | \ [g_1, {g}_2^{\varphi}]^{g_3} = [g_1, {g}_2^{\varphi}]^{{g}_3^{\varphi}} = [g_1^{g_3}, (g_2^{g_3})^{\varphi}] \hbox{ for } g_1, g_2, g_3 \in G \rangle$$
In the above presentation of $\nu(G)$ in terms of generators and relations $G^{\varphi}$ is an isomorphic copy of $G$  with an isomorphism $G \to G^{\varphi}$ that sends $g$ to $g^{\varphi}$.
The subgroup $[G, {G}^{\varphi}] \leq \nu(G)$ is isomorphic to the non-abelian tensor product $G \otimes G$.

 \begin{theorem} \cite[Thm. 2]{B-F-M} \label{wedge}  Let $G$ be a group and $F$ be a free group such that $G \simeq F/ R$ for some normal subgroup $R$ of $F$. Then $G \wedge G \simeq F'/ [F, R]$.
 \end{theorem} 
 
 \begin{theorem} \cite[Thm. 1]{B-F-M} \label{2tor}
 Let $G$ be a group such that $G/ G'$ is finitely generated. If $G/ G'$ has no
elements of order 2, or if $G'$ has a complement in $G$, then $G \otimes G \simeq \Delta(G) \times (G \wedge G)$.
 \end{theorem}
 
 We observe that in \cite{B-F-M} the authors use the notation $\nabla(G)$ for $\Delta(G)$.
 The condition on the absense of 2-torsion in $G/ G'$  in the statement of Theorem \ref{2tor} comes from the fact that in this case the map $\Gamma(G/G') \to \Delta(G/ G')$ is injective, where $\Gamma(G/ G')$ is the Whitehead quadratic functor, and this implies an isomorphism $\Delta(G/ G') \simeq \Delta(G)$.
 
 \subsection{Homology of discrete and profinite groups} Let $G$ be a discrete group. Then its profinite completion $\widehat{G}$ is defined as the inverse limit of its finite quotients  i.e.  $\widehat{G} \simeq   \varprojlim_U G / U$ where the inverse limit is over all normal subgroups $U$ of  finite index in $G$.  The completed group algebra  $$\widehat{\Z}[[\widehat{G}]] \simeq \varprojlim_{U,i}
\frac{\Z} {i \Z} [G/U]$$ where $\frac{\Z} {i \Z}[G/U]$ is the group algebra of $G/U$ with coefficients in $\Z/ i \Z$,  the inverse limit is over all finite subgroups $U$ of finite index in $G$ and all integers  $i \geq 2$.

If $B$ is a profinite $\widehat{\mathbb{Z}}$-module (= profinite abelian group)  i.e. $B =  \varprojlim_W (B/ W)$ where the inverse limit is over all subgroups $W$ of  finite index in the abelian group  $B$  
then by definition the completed (abelian) exterior product $$B \widehat{\wedge}_{\widehat{\Z}} B \simeq  \varprojlim_{ W} (B/ W) \wedge_{\mathbb{Z}} (B/W)$$ and the completed (abelian) tensor product $$B \widehat{\otimes}_{\widehat{\Z}} B \simeq  \varprojlim_{W} (B/ W) \otimes_{\mathbb{Z}} (B/W),$$ where we use $ \otimes_S$ and $ \wedge_S$ for the (abelian) tensor product and the (abelian) exterior product of $S$-modules.

More generally if $R$ is a profinite ring, $R = \varprojlim_{i,j} R_{i,j}$, where $R_{i,j}$ are finite quotients rings of $R$, $A$ is a right profinite $R$-module and $B$ is a profinite left $R$-module, such that $$A = \varprojlim_i A_i,  B = \varprojlim_j B_j$$ with $A_i, B_j$ finite $R$-module quotients of $A$ and $B$  such that $Ker(R \to R_{i,j})$ acts trivially on $A_i$ and $B_j$ then the completed (abelian) tensor product $$A \widehat{\otimes}_{R} B =  \varprojlim_{i,j} A_i \otimes_{R_{i,j}} B_j$$

We use $H_i(G, - )$ and $H_i(\widehat{G}, - )$ for homology of abstract and profinite groups, the later sometimes is called continuous homology. For basic properties we refer the reader to 
\cite{Rotman} and \cite{R-Z}. If $\widehat{G} = \varprojlim_U G/ U$  with $G/ U$ finite,  $B$ is  profinite $\widehat{G}$-module with  $B =  \varprojlim_W (B/ W)$ where each $B/ W$ is a finite $\widehat{G}$-module, hence for some normal subgroup  $U$ of finite index in $G$ we have that $U$ acts trivially on $B/ W$, we have
$$H_i(\widehat{G}, B ) \simeq  \varprojlim_{U, W} H_i(G/U, B/ W)$$

By \cite{Brown} if $G$ is an abelian (discrete) group then $H_2(G, \Z) \simeq G \wedge_{\Z} G$.

\begin{lemma} \cite{AT-BR-K} \label{wedge2}
Suppose $B$ is an abelian torsion-free profinite group  and $q \geq 2$ is an integer. Then 
 
 a) $H_2(B, \widehat{\Z}) \simeq B \widehat{\wedge}_{\widehat{\Z}}  B$,
 
 b)  $H_2(B, {\Z}/ q \Z) \simeq (B/ q B)  \widehat{\wedge}_{\widehat{\Z}} (B/ q B)$.
 \end{lemma}

\subsection{Criterions for residual finiteness} 

 Let $G$ be a discrete group. Recall that $\widehat{G}$ denotes the profinite completion of $G$ and $\widehat{G}_p$ the pro-$p$ completion of $G$. We state several results from \cite{AT-BR-K} that we will need later on.

\medskip \begin{theorem} \cite[Theorem A, a) ]{AT-BR-K} \label{A-L-Da} Let $G$ be a discrete group.

 1)  Suppose that $G$ is residually finite and the  map $$H_2(G, \mathbb{Z}) \to H_2(\widehat{G}, \widehat{\mathbb{Z}}),$$ induced by the canonical map  $G \to \widehat{G}$, is injective. Then $G \wedge G$ is residually finite.
 
 2)  Suppose that $G$ is residually finite and the  map $$H_2(G, \Z) \to \varprojlim_U H_2(G/U, \Z),$$ induced by the canonical projections $G \to G/ U$ (for $ \widehat{G} \simeq \varprojlim G/U$, $G/ U$ finite groups) is injective. Then $G \wedge G$ is residually finite.

\end{theorem} 
 \begin{theorem} \cite[Theorem A, c1), c2) ]{AT-BR-K} \label{A-L-Dc} Let $G$ be a discrete group   and $p$ be a prime integer.

      1) Suppose that $G$ is residually $p$-finite and the map $$H_2(G, \Z) \to \varprojlim_U H_2(G/U, \Z),$$ induced by the canonical projections $G \to G/ U$ (for $ \widehat{G}_p \simeq \varprojlim G/U$, $G/ U$ finite $p$-groups) is injective then $G \wedge G$ is residually $p$-finite.
 
 2)  Suppose that $G$ is residually $p$-finite and    the canonical map $$H_2(G, \mathbb{Z}) \to H_2(\widehat{G}_p, {\mathbb{Z}}_p),$$ induced by the canonical map  $G \to \widehat{G}_p$, is injective then $G \wedge G$ is residually $p$-finite.
 \end{theorem}

\begin{theorem} \cite[Theorem A, d) ]{AT-BR-K} \label{A-L-Dd} Let $G$ be a discrete group. If $G/ G'$ is finitely generated and   either has no 2-torsion or  $G'$ has a complement in $G$  then $G \otimes G$ is residually finite if and only if $G \wedge G$ is residually finite.
\end{theorem}

   The map $H_2(G, \mathbb{Z}) \to H_2(\widehat{G}, \widehat{\mathbb{Z}})$ in the above theorem is induced by the canonical maps $G \to \widehat{G}$ and $\Z \to \widehat{\Z}$. More precisely for $\widehat{G} \simeq \varprojlim_{i \in I} G_i$, where $G_i$ are finite quotients of $G$ and $\widehat{\Z} = \varprojlim_{j \in J} \Z/ m_j \Z$ then the homomomorphisms of groups $G \to G_i$ and $\Z \to \Z / m_j \Z$ induces a map in homology of abstract groups $H_2(G, \Z) \to H_2(G_i, \Z/ m_j )$. This gives a map
  $$H_2(G, \Z) \to H_2(\widehat{G}, \widehat{\Z}) \simeq \varprojlim_{i \in I,j \in J} H_2(G_i, \Z/ m_j \Z) 
  $$

 For $ q \geq 2$ an integer we set $\mathbb{F}_q = \Z/ q \Z$.

 \begin{prop} \cite[Prop. E] {AT-BR-K}\label{iff25} Let $q \geq 2$ be an integer,  $p$ be a prime integer  and $G$ be a finitely generated  discrete group.
 
 a)  Suppose that $G$ is residually finite.  Then $G \wedge^q G$ is residually finite if and only if the map $H_2(G, \mathbb{F}_q) \to H_2(\widehat{G}, \mathbb{F}_q)$, induced by the canonical map  $G \to \widehat{G}$, is injective.
 
 b) Suppose that $G$ is residually $p$-finite.  Then $G \wedge^q G$ is residually $p$-finite if and only if the map $H_2(G, \mathbb{F}_q) \to H_2(\widehat{G}_p, \mathbb{F}_q)$, induced by the canonical map  $G \to \widehat{G}_p$, is injective.
  \end{prop}

Note that this implies that
if $G$ is a finitely generated, residually finite group such that $G \in \mathcal{A}_2$ and $q \geq 2$ is an integer
then $G \wedge^q G$ is residually finite. For the definition of $\mathcal{A}_2$ we refer the reader to Section \ref{section-nongood}.

\section{Homology of wreath products} \label{Section3}

Let $$G = A \wr B$$ be the restricted wreath product i.e. $$G = M \rtimes B,$$ where $M = \oplus_{b \in B } A$ and $B$ acts via conjugation (on the right) on $M$ by permuting the factors via the action of $B$ on $B$ by multiplication on the right. Note that $A$ in general  is not abelian but we use the notation $\oplus_{b \in B} A$ to mean that for $a \in M$,  $a = (a_b)_{b \in B}$ where all but finitely many $a_b = 0$, all $a_b \in A$. For simplicity we write   $(a_b)$ for  $ (a_b)_{b \in B}$. Thus
an element $g \in G$ has a unique decomposition $g = a b$ where $b \in B$, $a \in M$. The conjugation action is given by $
(a_b) b_0 = b_0 (c_b)$ where $a_b =  c_{b b_0}$, $b_0 \in B$.

\begin{lemma} \label{L1}  Let $G = A \wr B$. Then
$H_2(G, \Z) \simeq H_0(B, H_2(M, \Z)) \oplus H_2(B, \Z)$.
\end{lemma}

\begin{proof}
Consider the LHS spectral sequence associated to the short exact sequence of groups $ 1 \to M \to G \to B \to 1$ i.e.
$$E_{p,q}^2 = H_p(B, H_q(M, \mathbb{Z}))$$
 that converges to $H_n(G, \mathbb{Z})$.
 Let $$i : B \to G$$
 be the inclusion map and $\pi : G \to B$ be the projection with kernel $M$. Then $\pi \circ i = id_B$, so $B$ is a retract of $G$. Let
 $$i_n : H_n(B, \mathbb{Z}) \to H_n(G, \Z) \hbox{ and } \pi_n : H_n(G, \mathbb{Z}) \to H_n(B, \Z)$$
 be the maps induced by  $i$ and  $\pi$ in homology. Then $id = (\pi \circ i)_n = \pi_n \circ i_n$, hence $i_n$ is injective and $\pi_n$ is surjective and $ H_n(B, \mathbb{Z})$ is a retract of $H_n(G, \Z)$.
 
 The convergence of the spectral sequence gives that $H_n: = H_n(G, \Z)$ has a filtration
 $$ 0 = \Phi_{-1}(H_n) \subseteq \Phi_{0}(H_{n}) \subseteq \ldots \subseteq \Phi_{k-1}(H_n) \subseteq \Phi_{k}(H_{n}) \subseteq \ldots \subseteq \Phi_n(H_n) = H_n$$
where 
$$\Phi_{k}(H_n)/ \Phi_{k-1}(H_n) \simeq E_{k, n-k}^{\infty}$$
By the construction of the spectral sequence  there are differential maps
$$d_{p,n-p}^r: E_{p, n-p}^r \to E_{p-r, n-p+r-1}^r$$
and $$E_{p, n-p}^{r+1} = Ker (d_{p,n-p}^r)/ Im (d_{p+r, n+1 - p - r}^r).$$
Thus
we have that $E_{k, n-k}^{\infty}$ is a subquotient of $E_{k, n-k}^{2}$, in particular using degree argument $E_{n, 0}^{\infty}$ is a subqroup of $E_{n, 0}^{2} = H_n(B, \Z)$ and the canonical projection map
$$H_n(G, \Z) \to \Phi_n(H_n)/ \Phi_{n-1}(H_n) \simeq E_{n,0}^{\infty}$$
is actually the map $\pi_n$.( To see this use that the spectral sequence is natural and compare it with the LHS spectral sequence $$\tilde{E}_{p,q}^2 = H_p(B, H_q(1, \mathbb{Z}))$$ corresponding to the short exact sequence of groups $1 \to 1 \to B \to B \to 1$ and link the two short exact sequences via the map $\pi$.) This implies that $E_{n,0}^{\infty} = E_{n,0}^2$ and all differentials $d_{n,0}^r : E_{n,0}^r \to E_{n-r, r-1}^r$ are zero. Then all differentials that leave or  end at $E_{1,1}^r$ are zero, so
$$E_{1,1}^2 = E_{1,1}^{\infty}$$
And for the same reason 
$$E_{2,0}^2 = E_{2,0}^{\infty}$$
We claim that 
$$E_{2,1}^2 = 0 = E_{1,1}^2, \hbox{ hence } E_{2,1}^{\infty} = 0 =  E_{1,1}^{\infty} $$ 
Then all differentials that enter and leave $E_{0,2}^r$ are zero, hence
$$E_{0,2}^2 = E_{0,2}^{\infty}$$
Then we conclude that there is a short exact sequence $0 \to E^{\infty}_{0,2} \to H_2(G, \mathbb{Z}) \to E_{2,0}^{\infty} \to 0$ i.e. $$0 \to H_0(B, H_2(M,  \mathbb{Z})) \to H_2(G, \mathbb{Z}) \to H_2(B, \mathbb{Z}) \to 0$$ Since  $H_2(B, \mathbb{Z})$ is a retract of $H_2(G, \mathbb{Z})$ we conclude that the  short exact sequence splits i.e. 
$$H_2(G, \mathbb{Z}) \simeq   H_0(B, H_2(M,  \mathbb{Z}))  \oplus H_2(B, \mathbb{Z})$$
To complete the proof we have to show that $E_{2,1}^2 = 0$  and $E_{1,1}^2 = 0$. Recall $$E_{2,1}^2 = H_2(B, H_1(M, \Z)) = H_2(B, M^{ab})$$
where $M^{ab} = M/ [M,M]$ is the abelianization of $M$. Note that
$$M^{ab} \simeq \oplus_{B} A^{ab} \simeq A^{ab} \otimes_{\Z} \mathbb{Z}  B$$
Then   by Shapiro Lemma   
$
E_{j,1}^2 \simeq H_j(B, M^{ab}) \simeq  { H_j(B, A^{ab} \otimes_{\Z} \mathbb{Z} B)  \simeq H_j(1, A^{ab}) =0}$ for $j \geq 1$. 
 \end{proof}
 
 Define $$C = A^{ab} \otimes_{\Z} A^{ab}$$ and denote $$  \mu : C \to C$$
 the $\Z$-linear map that swaps the two factors in the tensor product i.e. $\mu ( \bar{a}_1 \otimes \bar{a}_2) =  \bar{a}_2 \otimes \bar{a}_1$.
 Set
$$ Z(A, B) = \oplus_{b \in B \setminus \{ 1 \}}  Cb$$
where $Cb \simeq C$ and set $I(A,B)$ the $\Z$-submodule generated by $\{cb + \mu(c) b^{-1} \ | \ c \in C, b \in B \setminus \{ 1 \} \}$.

The following  theorem generalizes the main result of \cite{Bl}.
 \begin{theorem} \label{T1} There is an isomorphism of abelian groups
 $$H_2(A \wr B, \Z) \simeq H_2(B, \Z) \oplus H_2(A, \Z) \oplus ( Z(A, B) / I(A, B))$$
 \end{theorem}
 \begin{proof}
 By the K\"unneth formula $$H_2(A_1 \times A_2, \Z) \simeq H_2(A_1, \Z) \oplus H_2(A_2, \Z) \oplus H_1(A_1, \Z) \otimes_{\Z} H_1(A_2, \Z).$$
 Itterating we get
 $$H_2(A_1 \times \ldots \times A_n, \Z) \simeq (\oplus_{1 \leq i \leq n} H_2(A_i, \Z)) \oplus ( \oplus_{1 \leq i < j \leq n} H_1(A_i, \Z) \otimes_{\Z} H_1(A_j, \Z) )$$

 Note that $$M = \oplus_B A \simeq \oplus_{b \in B} A_b$$ where each $A_b \simeq A$ as a group and we fix a linear (total) order of $B$ as a set (that is not linked with the group action).
 Then
 since $H_2$ commutes with direct limits (over directed sets) we get that
 $$H_2(M, \Z) \simeq (\oplus_{b \in B} H_2(A_b, \mathbb{Z})) \oplus ( \oplus_{b_1 < b_2; b_1, b_2 \in B} H_1(A_{b_1}, \Z) \otimes_{\Z} H_1(A_{b_2}, \Z) )$$
 Note that
 $$\oplus_{b \in B} H_2(A_b, \mathbb{Z}) \simeq H_2(A, \Z) \otimes_{\Z} \Z B$$
  as a $\Z B$-module. Hence
  \begin{equation}
  \label{iso100} H_2(M, \Z) \simeq (H_2(A, \Z) \otimes_{\Z} \Z B) \oplus ( \oplus_{b_1 < b_2; b_1, b_2 \in B} H_1(A_{b_1}, \Z) \otimes_{\Z} H_1(A_{b_2}, \Z) ) \end{equation}
 Note that $$H_1(A_b, \Z) \simeq A_b^{ab} \hbox{ the abelianization of } A_b$$ and for $a_1 \in A_{b_1}, a_2 \in A_{b_2}$ with $b_1 < b_2$ in $B$ we have that for $\bar{a}_i $ the image of $a_i$ in $A_{b_i}^{ab}$ the element $\bar{a}_1 \otimes \bar{a}_2 \in H_1(A_{b_1}, \Z) \otimes_{\Z} H_1(A_{b_2}, \Z) $ corresponds to the element  
 $$r [R,F]  \in  (F' \cap R )  / [R, F] \simeq H_2(M, \Z)$$
 where $$ r = [a_1^{b_1}, a_2^{b_2}] = (a_1^{b_1})^{-1} (a_2^{b_2})^{-1} (a_1^{b_1}) (a_2^{b_2})  $$ is a relation of $M$ and we have used the Hopf formula for the Schur multiplier of $M$.
 Note that
 $$[a_1^{b_1} , a_2^{b_2}]^{-1} = [a_2^{b_2}, a_1^{b_1}] \hbox{ and } [a_1^{b_1}, a_2^{b_2}]^b = [a_1^{b_1 b}, a_2^{b_2 b}] $$
 Then there is an isomorphism of $\Z B$-modules
 \begin{equation} \label{iso200} 
 V(A, B) : = \oplus_{b_1 < b_2; b_1, b_2 \in B} H_1(A_{b_1}, \Z) \otimes_{\Z} H_1(A_{b_2}, \Z) \simeq \widetilde{V}(A,B)/ J\end{equation} where
 $$\widetilde{V}(A,B) : =
 (\oplus_{(b_1, b_2) \in B \times B \setminus diag(B)} (A^{ab} \otimes A^{ab}) (b_1, b_2))
 $$
 where  each $(A^{ab} \otimes A^{ab}) (b_1, b_2)$ is a copy $A^{ab} \otimes A^{ab}$ indexed by $(b_1, b_2)$,  the $B$ action on $\widetilde{V}(A, B)$ is the diagonal action of $B$ on $B \times B$ and $J$ is the $\Z$-submodule generated by $ \bar{a}_1 \otimes \bar{a}_2 (b_1, b_2) + \bar{a}_2 \otimes \bar{a}_1 (b_2, b_1)$.
 
 We want to calculate $H_0(B, V(A, B))$. The short exact sequence $0 \to J \to \widetilde{V}(A,B) \to V(A, B) \to 0$ gives an exact sequence
 $$H_0(B, J) \to H_0(B, \widetilde{V}(A,B)) \to H_0(B, V(A,B)) \to 0$$
 Note that
 $$H_0(B, \widetilde{V}(A,B)) \simeq \oplus_{b \in B \setminus \{ 1 \}} A^{ab} \otimes A^{ab} (1, b)$$
 that sends the class of $ \bar{a}_1 \otimes \bar{a}_2 (b_1, b_2)$ to  $\bar{a}_1 \otimes \bar{a}_2 (1, b_2 b_1^{-1})$. Under this isomorphism the image of $H_0(B,J)$ in  $H_0(B, \widetilde{V}(A,B))$ is sent to $ \bar{a}_1 \otimes \bar{a}_2 (1, b_2 b_1^{-1}) + \bar{a}_2 \otimes \bar{a}_1 (1, b_1 b_2^{-1})$.
  Hence
 \begin{equation} \label{iso300} H_0(B, V(A,B)) \simeq Z(A, B) / I(A, B) \end{equation}
Finally by  (\ref{iso100}), (\ref{iso200}) and (\ref{iso300})
$$H_0(B, H_2(M, \Z)) \simeq H_0(B, H_2(A, \Z) \otimes_{\Z} \Z B) \oplus H_0(B, V(A,B)) $$ $$\simeq 
H_2(A, \Z) \oplus ( Z(A, B) / I(A, B))$$

\end{proof}
 \begin{cor} \label{T1Cor} For an integer $q \geq 2$ there is an isomorphism of abelian groups
 $$H_2(A \wr B, \Z/ q \Z) \simeq H_2(B, \Z/q \Z) \oplus H_2(A, \Z/ q \Z) \oplus ( Z_q(A, B) / I_q(A, B))$$
 where $Z_q(A, B) = Z(A^{ab} \otimes  \Z/ q \Z, B)$ and $I_q(A, B) = I(A^{ab} \otimes  \Z/ q \Z, B)$.
 \end{cor}
 
 \begin{proof} By the universal coefficients theorem and the fact that  $H_1(A \wr B, \Z) \simeq A^{ab} \oplus B^{ab}$  there are exact sequences,   where $\otimes = \otimes_{\mathbb{Z}}$,
 $$0 \to H_2(A \wr B, \Z) \otimes  \Z/ q \Z \to H_2(A \wr B, \Z/ q \Z) \to Tor^{\Z}_1( A^{ab} \oplus B^{ab}, \Z/ q \Z) \to 0, $$
  $$0 \to H_2(A , \Z) \otimes  \Z/ q \Z \to H_2(A , \Z/ q \Z) \to Tor^{\Z}_1( A^{ab}, \Z/ q \Z) \to 0 $$
  and 
  $$0 \to H_2(B , \Z) \otimes  \Z/ q \Z \to H_2(B , \Z/ q \Z) \to Tor^{\Z}_1( B^{ab}, \Z/ q \Z) \to 0. $$
 Note that $$Tor^{\Z}_1( A^{ab} \oplus B^{ab}, \Z/ q \Z) \simeq Tor^{\Z}_1( A^{ab}, \Z/ q \Z) \oplus Tor^{\Z}_1(B^{ab}, \Z/ q \Z)$$ and $$Z_q(A, B) / I_q(A, B) \simeq (Z(A, B) / I(A, B)) \otimes (\Z/ q \Z)$$ This together with the above exact sequences and Theorem \ref{T1} implies Corollary \ref{T1Cor}. 
 \end{proof}

\section{Residual finiteness of $G \wedge G$ for $G = A \wr B$} \label{ChapterResFinWr}

By Gruenberg \cite{Gru}  $A \wr B$ is residually finite if and only if one of the following conditions holds:

\medskip
1) $A$ is abelian, $B$ and $A$ are residually finite;

\medskip
2) $A$ is residually finite, $B$ is finite.

\begin{lemma} \label{LL4}
Suppose that $G = A \wr B$ is residually finite. Then $\widehat{G}$ is the inverse limit of finite groups of  the type $A_i \wr B_i$, for some finite quotients $A_i$ and $B_i$ of $A$ and $B$ such that $\widehat{A} = \varprojlim_i A_i$ and $\widehat{B} = \varprojlim_i B_i$.
\end{lemma}

\begin{proof} Note that under the hypothesis $G$ is residually finite. Hence either $A$ is abelian or $B$ is finite.

Suppose that $K$ is a normal subgroup of finite index in $G$. Then $(A / (A \cap K)) \wr (B / (B \cap K))$ is a finite quotient of $G$. Set $J(K)$ to be the kernel of the canonical projection $G \to (A / (A \cap K)) \wr (B / (B \cap K))$. Thus $J(K)$ is the normal subgroup of $G$ 
generated by $A \cap K$ and $B \cap K$. Thus $J(K) \subseteq K$.

Let $\mathcal{F}$ be the set of normal subgroups of finite index in $G$. Then $\widehat{G}$ is the inverse limit $\varprojlim_{ K} G/ K$ over $K \in \mathcal{F}$, hence   $\widehat{G}$ is the inverse limit $G/ J(K)$ over $K \in \mathcal{F}$. Note that $G/ J(K) \simeq (A / (A \cap K)) \wr (B / (B \cap K))$ and both $A/ (A \cap K)$ and $B/ (B \cap K)$ are finite. We write $\{ A_i \}_i$ and $\{ B_i \}_i$ for  $\{A/ A \cap K \}_{K \in \mathcal{F}}$ and  $\{B/ B \cap K \}_{K \in \mathcal{F}}$.

Since $B$ is a retract of $G$ we have that the profinite topology of $G$ induces the profinite topology of $B$, hence $\widehat{B} = \varprojlim_i B_i$.

If $B$ is finite then $M$ is a direct product of finitely many copies of $A$ and the profinite topology of $G$ induces the profinite topology of $M$. Hence the profinite topology of $G$ induces the profinite topology of $A$.

Finally we show that if $A$ is abelian then the profinite topology of $G$ induces the profinite topology of $A$. Indeed let $N$ be a normal subgroup of $A$ of finite index. Note that $M \simeq A \otimes_{\mathbb{Z}} \mathbb{Z} B$ and let $\pi_1 : A \to A/ N$ be the canonical projection and $\pi_2 : \mathbb{Z} B \to \mathbb{Z}$ be the augmentation map. Then $\pi = \pi_1 \otimes \pi_2 : M \to  (A / N) \otimes_{\mathbb{Z}} \mathbb{Z} \simeq A/ N$ is a homomorphism of $\mathbb{Z} B$-modules, hence $ker(\pi)$ is a normal subgroup of $G$ that has finite index in $M$. We set $K = ker(\pi) \rtimes B$ and note that $K$ is a normal subgroup of finite index in $G$ such that $K \cap A = N$.

The previous two paragraphs imply that $\widehat{A} = \varprojlim_i A_i$.
\end{proof}

\begin{cor} \label{LL5} Under the hypothesis of the previous lemma

a)$$H_2( \widehat{G}, \widehat{\Z}) = \varprojlim_{i, m} H_2(A_i \wr B_i, \Z/ m \Z)$$
where  $\widehat{\Z} =  \varprojlim_{m} \Z / m \Z$ and  $\widehat{G} \simeq  \varprojlim_i  A_i \wr B_i$.

b) For an integer $q \geq 2$ we have $$H_2( \widehat{G}, {\Z}/ q \Z) \simeq \varprojlim_i H_2(A_i \wr B_i, \Z/ q \Z) \simeq$$ $$  H_2( \widehat{B}, \Z/ q \Z) \oplus  H_2( \widehat{A}, \Z/ q \Z) \oplus( \varprojlim_i Z_q(A_i, B_i) / I_q(A_i, B_i)) $$
\end{cor}

\begin{proof}  
a)  Follows directly by profinite homology commutes with inverse limits.

 b) The first isomorphism follows immediately from the fact  that the profinite homology commutes with inverse limits. The second follows from $$\varprojlim_i H_2(A_i \wr B_i, \Z/ q \Z) \simeq \varprojlim_i(  H_2(B_i, \Z/q \Z) \oplus H_2(A_i, \Z/ q \Z) \oplus ( Z_q(A_i, B_i) / I_q(A_i, B_i)))$$ $$ \simeq (\varprojlim_i  H_2(B_i, \Z/q \Z)) \oplus (\varprojlim_i H_2(A_i, \Z/ q \Z)) \oplus ( \varprojlim_i Z_q(A_i, B_i) / I_q(A_i, B_i)) \simeq $$ $$ H_2( \widehat{B}, \Z/ q \Z) \oplus  H_2( \widehat{A}, \Z/ q \Z) \oplus( \varprojlim_i Z_q(A_i, B_i) / I_q(A_i, B_i)) $$  
\end{proof}

Recall that for a $R$-module $V$ we write $V \square_{R}  V$ for $V \otimes_{R} V / \langle a \otimes b + b \otimes a \ | \ a,b \in V \rangle$.  For $W = \varprojlim_i W_i$ with $W_i$ finite, $W$ is a profinite abelian group, i.e., a profinite $\widehat{\Z}$-module, we write $W \widehat{\square}_{\widehat{\Z}} W$  for $ \varprojlim_i W_i {\square}_{\Z} W_i$.

\begin{theorem} \label{important} (= Theorem A)
Suppose that $G = A \wr B$ is a residually finite group with  $A$  an abelian group  such that:

1) the map  $  A \wedge_{\Z} A  \to  \widehat{A} \widehat{\wedge}_{\widehat{\Z}} \widehat{A}   $, induced by the canonical map $A \to \widehat{A}$, is injective;

2)
 $A \otimes_{\Z} A \to \widehat{A} \widehat{\otimes}_{\widehat{\Z}} \widehat{A}$  is injective   if $B$ has an element $b$ of order $|b| \in  \{ 3,4, \ldots \} \cup \{ \infty \}$  ;

3)  $A \square_{\Z} A \to \widehat{A} \widehat{\square}_{\widehat{\Z}} \widehat{A}$ is injective   if $B$ has an element $b$ of order   $|b| = 2$;

 4) the map  $H_2(B, \mathbb{Z}) \to \varprojlim_i H_2( B_i,{\Z})$ is injective where $\varprojlim_i B_i = \widehat{B}$ with $B_i$ finite.

Then the map $H_2(G, \mathbb{Z}) \to  \varprojlim_i H_2 (A_i \wr B_i, \Z)$ is injective, where $\widehat{G} =\varprojlim_i A_i \wr B_i$, $\widehat{A} =  \varprojlim_i A_i$, $\widehat{B} =  \varprojlim_i B_i$.   In particular,  $G \wedge G$ is residually finite.
\end{theorem}

 {\bf Remark} Note that if $ \widehat{G} \simeq \varprojlim_U G/ U$ with $G/ U$ finite then  $\varprojlim_U H_2(G/ U, \Z)$ depends only on $ \widehat{G}$ i.e. if $\widehat{G} \simeq \varprojlim_V G/ V$ with $G/ V$ finite then   $\varprojlim_V H_2(G/ V, \Z) \simeq \varprojlim_{V, U} H_2(G/ U \cap V, \Z) \simeq \varprojlim_U H_2(G/ U, \Z) $. 

\begin{proof}
 We will show that
\begin{equation} \label{inj1} H_2(G, \Z) \to \varprojlim_i H_2(A_i \wr B_i, \Z) \end{equation}
is injective,  where  by Lemma \ref{LL4} $ \widehat{G} \simeq \varprojlim_i A_i \wr B_i$  and  the above map is induced by the map $H_2(G, \Z) \to H_2(A_i \wr B_i, \Z)$ that itself is induced by the epimorphism $G \to A_i \wr B_i$. Then $G \wedge G$ is residually finite by Theorem \ref{A-L-Da}, 2).

By Theorem \ref{T1} we have that
 $$H_2(A \wr B, \Z) \simeq H_2(B, \Z) \oplus H_2(A, \Z) \oplus ( Z(A, B) / I(A, B))$$
and
 $$H_2(A_i \wr B_i, \Z) \simeq H_2(B_i, \Z) \oplus H_2(A_i, \Z) \oplus ( Z(A_i, B_i) / I(A_i, B_i))$$

Thus to prove that the map in (\ref{inj1}) is injective is equivalent to the following three conditions:

\medskip
 1) $ H_2(A, \Z) \to \varprojlim_i H_2(A_i , \Z)$ is injective,

\medskip
2)   $ H_2(B, \Z) \to \varprojlim_i H_2(B_i , \Z)$ is injective, 

\medskip
3) $\rho : Z(A, B) / I(A, B) \to \varprojlim_i Z(A_i, B_i) / I(A_i, B_i)$ is injective.

\medskip

We observe that since $A$ is abelian by Brown's book \cite[Thm. 6.4]{Brown}  $H_2 (A, \Z) \simeq A \wedge_{\Z} A$ the exterior power and $H_2(A_i, \Z) \simeq A_i \wedge_{\Z} A_i$. Then 
$$\varprojlim_i  H_2(A_i , \Z) =  \varprojlim_i A_i \wedge_{\Z} A_i = \widehat{A} \widehat{ \wedge}_{\widehat{\Z}} \widehat{A}    $$
the profinite exterior square of $A$.  Note that conditions 1) and 2) are given in the statement of our theorem.

Recall that $C = A^{ab} \otimes A^{ab} = A \otimes A$, $Z(A, B) = \oplus_{b \in B \setminus \{ 1 \}} C b$ and $I(A, B)$ is the $\Z$-submodule generated by $\{ c b + \mu(c) b^{-1} \ | b \in B \setminus \{ 1 \}, c \in C \}$, where $\mu$ is the  $\Z$-linear  map $C \to C$ that swaps the factors in the tensor product.
Then 
\begin{equation} \label{final-eq} Z(A, B)/ I(A, B) \simeq (\oplus_{ \{ b, b^{-1} \} \subseteq B, b \not= b^{-1} } Cb) \oplus ( \oplus_{b \in B, |b| = 2 }   (A \square_{\Z} A)b) \end{equation} 
 where  $(A \square_{\Z} A) b $ is a copy of $ A \square_{\Z} A$ indexed by $b$, and
 $$\varprojlim_i Z(A_i, B_i)/ I(A_i, B_i) \simeq \varprojlim_i ((\oplus_{ \{ b_i, b_i^{-1} \} \subseteq B_i, b_i \not= b_i^{-1} } C_ib_i) \oplus  ( \oplus_{b_i \in B_i, |b_i| = 2 } (A_i \square_{\Z} A_i)b_i)) $$
 where $C_i = A_i \otimes_{\Z} A_i$.
 
 We claim first that the map
 $$\rho_1 : ( \oplus_{b \in B, |b| = 2 }   (A \square_{\Z} A)b)  \to \varprojlim_i ( \oplus_{b_i \in B_i, |b_i| = 2 }  (A_i \square_{\Z} A_i)b_i)$$
 that is a restriction of $\rho$, is injective. For this it is enough to show that for any finite subset $X$ of $\{ b \in B \ | \ |b| = 2 \}$ the restriction $\rho_{1, X}$ of $\rho_1$ to  $\oplus_{b \in X }  (A \square_{\Z} A) b$ is injective.
 For any such $x \in X$ write $x_i$ for its image in $B_i$ and note that there exists $i$ such that the images of the elements of $X$ in $B_i$ are all different. Then restricting the codomain we can write
  $$\rho_{1, X} :  \oplus_{x \in X } (A \square_{\Z} A)x \to \oplus_{x \in X} (\varprojlim_i (A_i \square_{\Z} A_i) x_i) \simeq \oplus_{x \in X} \widehat{A} \widehat{\square}_{\widehat{\Z}} \widehat{A}$$
 Since $(A \square_{\Z} A)x$ is just a copy of $A \square_{\Z} A$, the injectivity of $\rho_{1,X}$ is equivalent to the map $A \square_{\Z} A  \to  \widehat{A} \widehat{\square}_{\widehat{\Z}} \widehat{A}$ is injective. 
 
Finally it remains to consider the map
$$ \rho_2:  \oplus_{ \{ b, b^{-1} \} \subseteq B, b \not= b^{-1} } Cb \to \varprojlim_i ((\oplus_{ \{ b_i, b_i^{-1} \} \subseteq B_i, b_i \not= b_i^{-1} } C_ib_i)  \oplus  ( \oplus_{b_i \in B_i, |b_i| = 2 }   (A_i \square_{\Z} A_i)b_i)) $$
that is a restriction of $\rho$. Since $B$ is residually finite if $b \in B$ has order $|b| \not= 2$ then there exists $i$ such that for the image $b_i$ of $b$ in $B_i$ we have that $|b_i | \not= 2$. 
Let $Y$ be an arbitrary finite subset of $\{ \{ b, b^{-1} \} \ | \ b \not= b^{-1} \}$. Note that we can find $i_0$ such that for $i \geq i_0$ the images of the elements $y = \{ b, b^{-1} \}$  of $Y$ are  subsets $ y_i = \{ b_i, b_i^{-1} \}$ of $B_i$ that are pairwise disjoint and $|b_i| > 2$.
Then $\rho_2$ induces a map
$$\rho_{2,Y} : \oplus_{ y = \{ b, b^{-1} \} \in Y } Cb \to  \oplus_{ y  \in Y} \varprojlim_{i \geq i_0} C_ib_i \simeq \oplus_{ y  \in Y}   (\widehat{A} \widehat{\otimes}_{\widehat{\Z}} \widehat{A})  $$
Since $Cb$ is a copy of $C$ the injectivity of  $\rho_{2, Y}$ is equivalent to the fact that the map $C = A \otimes_{\Z} A \to \widehat{A} \widehat{\otimes}_{\widehat{\Z}} \widehat{A}  $ is injective. 

Finally since  $\rho_1$, $\rho_{2,Y}$ (for every finite $Y$) are injective, we conclude that the map $\rho$ is injective.
\end{proof}

 Set  $D = A^{ab} \otimes_{\Z} (\Z/ q \Z)$ where $A$ is an abelian group.

\begin{theorem} \label{q-wreath-abelian}
Suppose that $G = A \wr B$ is a residually finite group with  $A$  an abelian group  such that:

1) the map $H_2(A, \mathbb{Z}/ q \Z)  \to H_2( \widehat{A}, {\Z}/ q \Z) $ is injective;

2)  $D \otimes_{\Z} D \to \widehat{D} \widehat{\otimes}_{\widehat{\Z}} \widehat{D}$ is injective   if $B$ has an element $b$ of order   $|b| \in  \{ 3,4, \ldots \} \cup \{ \infty \}$ 
and  $D \square_{\Z} D \to \widehat{D} \widehat{\square}_{\widehat{\Z}} \widehat{D}$ is injective   if $B$ has an element $b$ of order $|b| = 2$;

3) the map  $H_2(B, \mathbb{Z}/ q \Z) \to H_2( \widehat{B},{\Z}/ q \Z)$ is injective.

Then the map $H_2(G, \mathbb{Z}/ q \Z) \to H_2( \widehat{G}, {\Z}/ q \Z)$ is injective.  In particular $G \wedge^q G$ is residually finite.
\end{theorem}

\begin{proof}
We can apply Corollary \ref{T1Cor} and Corollary \ref{LL5} to calculate $H_2(G, \mathbb{Z}/ q \Z)$ and  $H_2( \widehat{G}, {\Z}/ q \Z)$.
It remains only to verify that condition 2) implies that the map $ Z_q(A, B) / I_q(A, B) \to \varprojlim_i  Z_q(A_i, B_i) / I_q(A_i, B_i)$ is injective. Since $$ Z_q(A, B) / I_q(A, B) \simeq  Z(D, B) / I(D, B)$$
by the arguments of the proof of Theorem \ref{important}  the injectivity we want to prove follows by condition 2). 
\end{proof}

\begin{lemma} If $A$ is a finitely generated abelian group then the canonical maps  $A \wedge_{\Z} A \to \widehat{A} \widehat{\wedge}_{\widehat{\Z}} \widehat{A}$, 
$A \otimes_{\Z} A \to \widehat{A} \widehat{\otimes}_{\widehat{\Z}} \widehat{A}$ and $A \square_{\Z} A \to \widehat{A} \widehat{\square}_{\widehat{\Z}} \widehat{A}$   are  injective.
\end{lemma}

 \begin{proof} It follows from the fact that $A \simeq \Z^n \oplus T$, where $T$ is a finite abelian group. \end{proof} 

\begin{cor}
Let $B$ be a residually finite group such that   $H_2(B, \Z) \to \varprojlim_j H_2(B_j, \Z)$  is injective, where $\widehat{B} = \varprojlim_j B_j$, $B_j$ finite quotients of $B$. Let $\{ A^{(i)} \}_{i \geq 1}$ be finitely generated abelian groups. Set inductively $G_1 = A^{(1)} \wr B$, $G_{i+1} = A^{(i+1)} \wr G_i$. Then 

a) for every $i$ we have that the map $H_2(G_i, \Z) \to  \varprojlim_j H_2( G_{i,j},   {\Z} )$ is injective, where $\widehat{G}_i = \varprojlim_j G_{i,j}$ and   $G_{i,j}$ are finite quotients of $G_i$.

b) the non-abelian exterior product $G_i \wedge G_i$ is residually finite.

c) If furthermore each of the groups $A^{(1)}, \ldots, A^{(i)}, B^{ab}$ does not have 2-torsion  then $G_i \otimes G_i$ is residually finite.
\end{cor}
\begin{proof}
a) follows by induction on $i$ using Theorem \ref{important}.

b) follows from a) and Theorem \ref{A-L-Da}, 2). 

c) follows   from the fact  that the abelianization $G_i^{ab}$ does not have 2-torsion, hence we can apply 
Theorem \ref{A-L-Dd}, d) 

\end{proof} 

\begin{theorem} \label{oldThmB}
Suppose $G = A \wr B$ is a group, $A$ is  a residually finite group, $B$  is a  finite group such that:

1) the map  $H_2(A, \Z) \to \varprojlim_i H_2( A_i ,{\Z})$ is injective, where $\widehat{A} \simeq \varprojlim_i A_i$  with $A_i$  finite quotients of $A$,

2)  if $B$ has an element $b$ of order   $|b| \in  \{ 3,4, \ldots \}$   suppose that the map
 $A^{ab} \otimes_{\Z} A^{ab} \to \widehat{A}^{ab} \widehat{\otimes}_{\widehat{\Z}} \widehat{A}^{ab}$ is injective,

3) if $B$ has an element $b$ of order   $|b| = 2$   suppose that the map
 $A^{ab} \square_{\Z} A^{ab} \to \widehat{A}^{ab} \widehat{\square}_{\widehat{\Z}} \widehat{A}^{ab}$ is injective.

Then the map  $H_2(G, \Z) \to \varprojlim_U  H_2( G/U, {\Z})$ is injective, where $\widehat{G} \simeq \varprojlim_U G/ U$ with $G/ U$ finite and consequently  $G \wedge G$ is residually finite.
\end{theorem}

\begin{proof}$A^{ab} \square_{\Z} A^{ab} \to \widehat{A}^{ab} \widehat{\square}_{\widehat{\Z}} \widehat{A}^{ab}$
We just repeat the proof of Theorem \ref{important} with slight modifications. First substitute  $A \wedge_{\Z} A$, $A \otimes_{\Z} A$,  $A \square_{\Z} A$,  $\widehat{A} \widehat{\otimes}_{\widehat{\Z}} \widehat{A}$,  $\widehat{A} \widehat{\wedge}_{\widehat{\Z}} \widehat{A}$  and  $\widehat{A} \widehat{\square}_{\widehat{\Z}} \widehat{A}$    with $A^{ab} \wedge_{\Z} A^{ab}$, $A^{ab} \otimes_{\Z} A^{ab}$, $A^{ab} \square_{\Z} A^{ab}$, $\widehat{A}^{ab} \widehat{\otimes}_{\widehat{\Z}} \widehat{A}^{ab}$, $\widehat{A}^{ab} \widehat{\wedge}_{\widehat{\Z}} \widehat{A}^{ab}$ and  $\widehat{A}^{ab} \widehat{\square}_{\widehat{\Z}} \widehat{A}^{ab}$ respectively. Second since $B$ is finite, we have $\widehat{B} = B$, this simplifies the proof. Note that $G = (\oplus_B A) \rtimes B$.
\end{proof}
 
The previous result has a version with coefficients in $\z/ q \Z$ that can be proved by modifying the proof of Theorem \ref{q-wreath-abelian}.

\begin{theorem} \label{q-section-finite} Let $q \geq 2$ be an integer. 
Suppose that $G = A \wr B$ is a residually finite group with $A$  residually finite, $B$ finite, $D =   A^{ab}/ q A^{ab} $  such that:

1) the map $H_2(A, \mathbb{Z}/ q \Z)  \to H_2( \widehat{A}, {\Z}/ q \Z) $ is injective;

2) if $B$ has an element $b$ of order   $|b| \in  \{ 3,4, \ldots \}$  suppose that the map
$D \otimes_{\Z} D \to D \widehat{\otimes}_{\widehat{\Z}} D$ is injective;

3) if $B$ has an element $b$ of order   $|b|=2$ suppose that the map
$D \square_{\Z} D \to D \widehat{\square}_{\widehat{\Z}} D$ is injective.

Then the map $H_2(G, \mathbb{Z}/ q \Z) \to H_2( \widehat{G}, {\Z}/ q \Z)$ is injective.  In particular $G \wedge^q G$ is residually finite.
\end{theorem}

\section{Residual finiteness of $G \wedge G$ for the combinatorial wreath product $G = A \wr_X B$ with $B$ a free group} \label{ChapterFree}

 The restricted wreath product construction  $ A \wr B$ can be generalized to the  combinatorial  wreath product $A \wr_X B$, where $X$ is a set acted (on the right) by the group $B$. Here $M$ is the normal closure of $A$ in $G$ i.e. $M = A^{(X)} = \oplus_X A$ and a typical element $a \in M$ is of the form  $a = (a_x)_{x \in X}$ where all but finitely many $a_x = 0$, all $a_x \in A$. Then $
(a_x) b_0 = b_0 (c_x)$ where $a_x =  c_{x b_0}$, $b_0 \in B$.

\medskip
By \cite[Prop. 3.2]{C2}    if $X = \cup_i  H_i \backslash B$ is  a disjoint union of $B$-orbits, then   $A \wr_X B$ is residually finite if and only if:

a) both $A$ and $B$ are residually finite;

b) one of the following conditions holds:

b1) $H_i$ has finite index in $B$ for all $i$

or

b2) $H_i$ is profinitely closed in $B$ for all $i$ and $A$ is abelian.

In b2)  “profinitely closed” means closed in the profinite topology, which for a
 subgroup $K$ of $B$  means that  $K = \cap_U KU$ where the intersection is over all normal subgroups $U$ of finite index in $B$.

\medskip
\begin{lemma} \label{res-fin} 
Suppose $G = A \wr_X B$ is a  residually finite group such that $B$ acts on $X$ with finitely many orbits.
 Then

1)  $\widehat{G} = \varprojlim_i G_i$, where  $G_i = A_i \wr_{X_i} B_i$, for some finite quotients $A_i$,$B_i$ and $X_i$ of $A$,$B$ and $X$ such that  $\widehat{A} = \varprojlim_i A_i$,  $\widehat{B} = \varprojlim_i B_i$ and the canonical map $X \to \varprojlim_i X_i$ is injective;

2)  $H_2( \widehat{G}, \widehat{\Z}) = \varprojlim_{i,m} H_2(G_i, \Z/ m \Z)$, where $ \widehat{\Z} = \varprojlim_m \Z / m \Z$.

\end{lemma}

\begin{proof} 1) By Cornulier's criterion for residual finiteness of $G$ \cite[Prop. 3.2]{C2} both $A$ and $B$ are residually finite and either every stabiliser in $B$ of a point of $X$ has finite index in $B$ or  $A$ is abelian and every stabiliser $H$ ( in $B$) of a point of $X$ is closed in the profinite topology of $B$.

Note that $X = \cup_{j \in J} H_j  \backslash B$ where the union is disjoint.  Since $B$ acts on $X$ with finitely many orbits the set $J$ is finite. Since $A$ is residually finite then $\widehat{A} = \varprojlim_{\delta \in \Delta} A_{\delta}$ with each $A_{\delta}$ finite quotient of $A$.

a) We suppose first that $A$ is abelian and every stabiliser $H$ ( in $B$) of a point of $X$ is closed in the profinite topology of $B$. Since each $H_j$ is closed in the profinite topology of $B$  there are   normal subgroups $K_{j,s} $  of finite index in $B$  such that  $H_j = \cap_{s \in S} H_j K_{j,s}$,  $\cap_{s \in S} K_{j,s} = 1$  and $\widehat{B} = \varprojlim_{s} B/ K_{j,s}$ for all $j \in J$. 

We set $I = \Delta \times S$, $$A_{(\delta,s)} = A_{\delta}, B_{(\delta,s)} = B/ \cap_{j \in J} K_{j,s} \hbox{  and }X_{(\delta,s)} = \cup_{ j \in J} (H_j K_{j,s} \backslash B)$$
where the union in $X_{(\delta,s)}$ is disjoint.  For $i = (\delta,s)$ we set $G_i = A_i \wr_{X_i} B_i$.
 We consider the map $\pi_i : X \to X_i$ that  sends a  $B$-orbit $H_j \backslash B$    to the orbit $H_j K_{j,s} \backslash B$. Note that if we consider the inverse system of all finite $B$-quotients of $X$ then its inverse  limit $\widehat{X} \simeq \varprojlim_{i \in I} X_i$. 
 The fact that the canonical map $X \to \varprojlim_i X_i$ is injective follows from $H_j = \cap_s H_j K_{j,s}$ for all $j$. Note that there is an epimorphism of groups $$G = A \wr_X B \to A_i \wr_{X_i} B_i$$ induced by the epimorphisms $A \to A_i$, $B \to B_i$ and the map $\pi_i : X \to X_i$.   We claim that
$$ \widehat{G} \simeq  \varprojlim_{i \in I} A_i \wr_{X_i} B_i$$
Since each $ A_i \wr_{X_i} B_i$ is a finite quotient of $G$ we need to show that every finite quotient $C$ of $G$ is a quotient of  $A_i \wr_{X_i} B_i$ for some $i$. 

Suppose $\theta : G \to C$ is an epimorphism. For $M$ the normal closure of $A$ in $G$ we have that $M \simeq A \otimes_{\mathbb{Z}} \mathbb{Z} X$ and define $I_x = \{ a \in A \ | \ \theta ( a \otimes x) = 1_C \}$ for $x \in X$, so $A/ I_x$ is isomorphic to a subgroup of $C$ and is finite. Recall that conjugation on $M$ (on the right) by $B$ corresponds to the action of $B$ (on the right) on $X$, hence $I_x = I_{x b}$. We claim that $\cap_{x \in X} I_x$ is a subgroup of finite index in $A$. Indeed since $B$ acts on $X$ with finitely many orbits, say $B$-orbits generated by $\{ x_j \ | \ j \in J \} $ then each $I_x = I_{x_j}$ for some $ j \in J$, hence   $\cap_{x \in X} I_x = \cap_{j \in J}  I_{x_j}$ has finite index in $A$.  Then we have to choose $i = (\delta,s)$ such that $A_i = A_{\delta}$ has as quotient $A / \cap_{x \in X} I_x$ i.e. this corresponds to the choice of $\delta$ and $B_i = B/ \cap_{j \in J} K_{j,s}$ has a quotient $\theta(B)$ i.e. this correspond to the choice of $s$. Then $C$ is a quotient of $A_i \wr_{X_i} B_i$.

 The choice of $K_{j,s}$ and $B_i$ implies  $\widehat{B} = \varprojlim_{i \in I} B_i$. 

 Let  $\overline{A}$ be a finite quotient of $A$. Take some $j \in I$ and consider $\overline{A} \wr_{X_j} B_j$. It is a finite quotient of $G$, so  by the previous paragraph $\overline{A} \wr_{X_j} B_j$  is a quotient of some $A_i \wr_{X_i} B_i$, in particular  $\overline{A}$ is a quotient of $A_i$. This implies $\widehat{A} = \varprojlim_{i \in I} A_i$. 

b) Now we consider the case when $A$ is not abelian   and each stabiliser $H_i$ has finite index in $B$, hence  $X$ is  finite.

Let $\widehat{B} = \varprojlim_{k \in K} B_k$ with $B_k$ finite quotients of $B$ and $\widehat{A} = \varprojlim_{\delta \in \Delta} A_{\delta}$ with $A_{\delta}$ finite quotients of $A$. Set $I = \Delta \times K$, for $i = (\delta, k)$ we set  $A_i = A_{\delta}$, $B_{i} = B_{k}, X_i = X$ and since $X$ is finite, $Perm(X)$ is finite and we can assume that for every $k \in K$ the kernel $Ker(B \to B_k)$ is inside the kernel of the action of $B$ on $X$.
 Then there is an epimorphism of groups $G = A \wr_X B \to A_i \wr_{X_i} B_i$ induced by the epimorphism $A \to A_i$, $B \to B_i$ and the identity map $id_X$. As before 
$$ \widehat{G} \simeq  \varprojlim_{i \in I} A_i \wr_{X_i} B_i$$

2)  Follows directly from the fact that profinite homology commutes with inverse limits.
\end{proof}

\begin{lemma} \label{discrete-homology}Let $G =  A \wr_X B$ be residually finite, where $B$ is a finitely generated free group, $A$ is a finitely generated abelian group. Let  $M = \oplus_{x \in X} Ax$, so $G = M \rtimes B$. Then  

a) there is a short exact sequence of abelian groups $ 0 \to H_0(B, M 
\wedge_{ \Z} M) \to H_2(G, \Z) \to H_1(B, M) \to 0$.

b) for every integer $q \geq 2$  there is a short exact sequence of abelian groups $ 0 \to H_0(B, H_2(M, \Z/ q \Z)) \to H_2(G, \Z/ q \Z) \to H_1(B, M/ qM) \to 0$.
 \end{lemma} 
\begin{proof} a)
Consider the LHS spectral sequence $$E_{i,j}^2 = H_i(B, H_j (M, \Z))$$ that converges to $H_{i+j} (G, \Z)$.  Since $B$ is a free group $E_{i,j}^2 = 0$ for all $i \geq 2$. 
  Then by comparing the bidegrees of the maps we conclude that $$E_{0,2}^{\infty} = E_{0,2}^2 = H_0(B, H_2(M, \Z)) \simeq  H_0(B, M \wedge_{\Z}  M)$$ and $$  E_{1,1}^{\infty} = E_{1,1}^2 = H_1(B, H_1(M, \Z)) \simeq H_1(B, M)$$
 Finally by the convergence of the spectral sequence  
 $$0 \to  E_{0,2}^{\infty}  \to H_2(G, \Z) \to E_{1,1}^{\infty} \to 0$$ 
 
  b) The proof is the same as the proof of a) but in the homology group change coefficients $\Z$ with $\Z/ q \Z$.
\end{proof}

Let $G =  A \wr_X B$ be residually finite, where $B$ is a finitely generated free group, $A$ is a finitely generated free abelian group,   $B$ acts on $X$ with finitely many orbits.  By Lemma \ref{res-fin} the profinite completion $\widehat{G} = \varprojlim_i G_i$, where  $G_i = A_i \wr_{X_i} B_i$, for some finite quotients $A_i$,$B_i$ and $X_i$ of $A$,$B$ and $X$ such that  $\widehat{A} = \varprojlim_i A_i$,  $\widehat{B} = \varprojlim_i B_i$ and the canonical map $X \to   \widehat{X} :=  \varprojlim_i X_i$ is injective. We view $\widehat{X}$ as a profinite space and set $\widehat{A}[[\widehat{X}]] : = \varprojlim_i A_i [X_i]$, where $A_i[X_i] = M_i$ is the normal closure of $A_i$ in $G_i$.

 Thus $\widehat{G} \simeq \widehat{A}[[\widehat{X}]] \rtimes \widehat{B}$. We write $\overline{M}$ for the closure of $M$ in $\widehat{G}$ i.e. $\overline{M} \simeq \widehat{A}[[\widehat{X}]]$.

 Recall  that $X = \cup_{j \in J} H_j  \backslash B$ where the union is disjoint and $J$ is a finite set.

\begin{lemma} \label{profinite-homology}Let $G =  A \wr_X B$ be a residually finite group, where $B$ is a finitely generated free group,  $A = \mathbb{Z}^s$  and  $B$ acts on $X$ with finitely many orbits.   Then  

a) there is a short exact sequence of abelian groups $$ 0 \to H_0(\widehat{B},    \overline{M} \widehat{\wedge}_{\widehat{\Z}} \overline{M}) \to H_2(\widehat{G}, \widehat{\Z}) \to H_1(\widehat{B}, \overline{M}) \to 0$$ 

 b) there is a short exact sequence of abelian groups $$ 0 \to H_0(\widehat{B},    H_2(\overline{M}, \Z/ q \Z)) \to H_2(\widehat{G}, {\Z}/ q \Z) \to H_1(\widehat{B}, \overline{M}/ q \overline{M}) \to 0$$

\end{lemma}

\begin{proof} a) The proof is a profinite version of the proof of Lemma \ref{discrete-homology}.
Consider the LHS spectral sequence $$E_{i,j}^2 = H_i(\widehat{B}, H_j (\overline{M}, \widehat{\Z}))$$ that convergence to $H_{i+j} (\widehat{G}, \widehat{\Z})$. Since $B$ is a finitely generated free group, $\widehat{B}$ is a free profinite group. Then $E_{i,j}^2 = 0$ for $i \geq 2$.

  Then by comparing the bidegrees of the maps and by Lemma \ref{wedge2}  we conclude that $$E_{0,2}^{\infty} = E_{0,2}^2 = H_0(\widehat{B}, H_2(\overline{M},  \widehat{\Z} )) \simeq  H_0(B, \overline{M} \widehat{\wedge} \overline{M})$$ where above we used Lemma \ref{wedge2}  and $$  E_{1,1}^{\infty} = E_{1,1}^2 = H_1(\widehat{B}, H_1(\overline{M},  \widehat{\Z} )) \simeq H_1(\widehat{B}, \overline{M})$$
  Finally  by the convergence of the spectral sequence there is a short exact sequence 
 $$0 \to   E_{0,2}^{\infty}  \to H_2(\widehat{G}, \widehat{\Z}) \to E_{1,1}^{\infty} \to 0$$
 
 b) In the proof of a) substitute coefficients $\widehat{\Z}$ with $\Z/ q \Z$.
\end{proof}

\begin{theorem} \label{comb-wreath-main}  Let $G =  A \wr_X B$ be a residually finite group, where $B$ is a finitely generated free group, $A$ is a torsion-free finitely generated abelian group, $B$ acts on $X$ with finitely many orbits  and each stabiliser $stab_B(x)$ is finitely generated for $x \in X$.  Then 

a) the map $H_2(G, \Z) \to H_2( \widehat{G}, \widehat{\Z})$
 is injective. In particular $G \wedge G$ is residually finite

b) for every integer $ q \geq 2$  the map $H_2(G, \Z/ q \Z) \to H_2( \widehat{G}, {\Z}/ q \Z) \simeq  \varprojlim_i H_2( A_i \wr_{X_i} B_i, \Z/  q \Z)$   is  injective. In particular $G \wedge^q G$ is residually finite.
\end{theorem}

\begin{proof} 
 Suppose $H_i = stab_B(x_i)$ where $X =  \cup_{i \in J} x_i B$ is a disjoint union   with $J$ finite. 
Let $\overline{H}_i$ be the closure of $H_i$ in $ \widehat{B}$. Note that since $H_i$ is finitely generated subgroup of a free group by \cite{Hall} $H_i$ is a virtual retract i.e. there is a subgroup of finite index $S_i$ of $B$ such that $H_i$ is a retract of $S_i$. Then $S_i$ is a finitely generated free group, the profinite topology of $B$ induces the profinite topology of $S_i$ and since $H_i$ is a retract of $S_i$ the profinite topology of $S_i$ induces the profinite topology of $H_i$. Then  $\overline{H}_i \simeq \widehat{H}_i$ and $H_i$ is a finitely generated free group, so $\widehat{H}_i$ is a finitely generated free profinite group.

Consider the short exact sequence
$ 0 \to H_0(B, M \wedge_{\Z} M) \to H_2(G, \Z) \to H_1(B, M) \to 0$ given by Lemma \ref{discrete-homology}.

\medskip
{\bf Claim 1} a)  $H_1(B, M) \to H_1( \widehat{B}, \overline{M})$ is injective.

 b) For $q \geq 2$ we have that $H_1(B, M/q M ) \to H_1( \widehat{B}, \overline{M}/  q\overline{M})$ is injective.

\medskip {\bf Proof of Claim 1} a) As before $J$ is the set of orbits of the action of $B$ on $X$.
 Here $M =  A \otimes_{\Z} \Z X$ and $\overline{M} = \widehat{A} [[ \widehat{X}]] \simeq \oplus_{j \in J} \widehat{A} \widehat{\otimes}_{ \widehat{\Z}[[ \overline{H}_j]]} \widehat{\Z}[[\widehat{ B}]]$.
Note that  $$H_1(B, M)  \simeq H_1(B, A \otimes_{\Z} \Z X) \simeq A \otimes_{\Z} H_1(B, \Z X) \simeq A \otimes_{\Z} (H_1(B,  \oplus_{j \in J} \Z \otimes_{\Z H_j} \Z B)) \simeq $$ 
$$A \otimes_{\Z} ( \oplus_{j \in J} H_1(H_j, \Z)) \simeq  A \otimes_{\Z}  ( \oplus_{j \in J} H_j^{ab}) \simeq \oplus_{j \in J} A \otimes_{\Z}   H_j^{ab}   $$
 and $$H_1( \widehat{B}, \overline{M}) \simeq  \oplus_{j \in J} H_1(\widehat{B}, \widehat{A} \widehat{\otimes}_{ \widehat{\Z}[[ \overline{H}_j]]} \widehat{\Z}[[\widehat{ B}]]) \simeq  \oplus_{j \in J} \widehat{A}  \widehat{\otimes}_{\Z} H_1(\widehat{B}, {\widehat{\Z}}   \widehat{\otimes}_{\widehat{\Z} [[\overline{H}_j]]} 
 \widehat{\Z}[[\widehat{ B}]])$$ $$ \simeq \oplus_{j \in J} \widehat{A}  \widehat{\otimes}_{\widehat{\Z}} H_1(\overline{H}_j, {\widehat{\Z}}) \simeq
  \oplus_{j \in J} \widehat{A}  \widehat{\otimes}_{\widehat{\Z}} \overline{H}_j^{ab} 
 $$
Recall that  that $\overline{H}_j \simeq \widehat{H}_j$ is a free finitely generated profinite group,  hence the abelianization $\overline{H}_j^{ab} \simeq \widehat{\mathbb{Z}}^{r_j}$ for some non-negative integer $r_j$. The fact that $A$ is finitely generated abelian implies that $A \to \widehat{A}$ is injective.  Note that the map $A \otimes_{\Z} H_j^{ab} \to \widehat{A} \otimes_{\widehat{\Z}}  \overline{H}_j^{ab} $ is injective.

 b) Similarly we get $$H_1(B, M/ q M)  \simeq  \oplus_{j \in J} (A/qA) \otimes_{\Z}  (  H_j^{ab})   \hbox{ and } H_1( \widehat{B}, \overline{M}  /  q  \overline{M} ) \simeq  \oplus_{j \in J} (\widehat{A}/ q \widehat{A})  \widehat{\otimes}_{\widehat{\Z}} \overline{H}_j^{ab} $$
This completes the proof of Claim 1.

\medskip
{\bf Claim 2} a) $H_0({B},    {M}{\wedge} {M}) \to H_0(\widehat{B},    \overline{M} \widehat{\wedge}_{\widehat{\Z}} \overline{M})$ is injective.

 b)  $H_0({B}, H_2(M, \Z/ q \Z)) \to H_0(\widehat{B},    H_2(\overline{M}, \Z/ q \Z))$ is injective. 

\medskip {\bf Proof of Claim 2} a)
  Set  $$Z =  X \times X \setminus diag(X)$$ and $\epsilon: Z \to Z$ sends $( x_1, x_2)$ to $(x_2, x_1)$. Define $\widetilde{Z} = Z/ \sim$ where  $y_1, y_2 \in Z$, $y_1 \sim y_2$ if $y_1 \in \{ y_2, \epsilon(y_2) \}$.

Recall that $M =  A \otimes_{\Z} \Z X \simeq \oplus_{x \in X} Ax$, where $Ax \simeq A$. Hence   $$M \wedge_{\Z} M \simeq (\oplus_{ \{ y = (x_1, x_2), \epsilon(y) \} \subset Z } Ax_1 \otimes_{\Z} A x_2) \oplus (\oplus_{x \in X} A x \wedge_{\Z} Ax) \simeq$$ $$ ((A \otimes_{\Z} A) \otimes_{\Z} \Z  \widetilde{Z}) \oplus ((A \wedge_{\Z} A) \otimes_{\Z} \Z X) $$
hence
 $$ H_0({B},    {M}{\wedge}_{\Z} {M}) \simeq 
((A \otimes_{\Z} A) \otimes_{\Z} H_0 (B, \Z  \widetilde{Z})) \oplus ((A \wedge_{\Z} A) \otimes_{\Z} H_0(B,\Z X))\simeq$$ $$ 
((A \otimes_{\Z} A) \otimes_{\Z}  \Z  (\widetilde{Z}/B)) \oplus ((A \wedge_{\Z} A) \otimes_{\Z} \Z (X/B))$$ 
Note that by Lemma \ref{res-fin} $ \widehat{G} = \varprojlim_i A_i \wr_{X_i} B_i$ with { $A_i, B_i, X_i =  \cup_{ j \in J} (H_j K_{j,i} \backslash B)$ finite, where $X = \cup_{j \in J} H_j \backslash B$.} Here $\overline{M} =  \varprojlim_i A_i \otimes_{\Z} \Z_{m_i}[X_i]$ where $m_i$ is the exponent of $A_i$. Then 
 $$ H_0(\widehat{B},    \widehat{M}\widehat{\wedge}_{\widehat{\Z}} \widehat{M}) \simeq 
 \varprojlim_i H_0(B_i,  (A_i \otimes_{\Z} \Z_{m_i} X_i) \wedge ( A_i \otimes_{\Z} \Z_{m_i}X_i)) \simeq$$ 
 $$ \varprojlim_i  ((A_i {\otimes}_{\Z} A_i) {\otimes}_{\Z} \Z_{m_i}  (\widetilde{Z}_{i}/B_i)) \oplus \varprojlim_j  ((A_i \wedge_{\Z} A_i) \otimes_{\Z} { \Z_{m_i} } (X_i/B_i)) $$
where $Z_i =  X_i \times X_i \setminus diag(X_i)$, $\epsilon_i: Z_i \to Z_i$ sends $( x_1, x_2)$ to $(x_2, x_1)$ and $\widetilde{Z}_{i} = Z_i/ \sim$ where  $y_1, y_2 \in Z_i$, $y_1 \sim y_2$ if $y_1 \in \{ y_2, \epsilon_i(y_2) \}$.

Since $A$ is a finitely generated abelian group then $A \otimes_{\Z} A$ embeds in $\widehat{A} \widehat{\otimes}_{\widehat{\Z}} \widehat{A}$ and $ A \wedge_{\Z} A$ embeds in $\widehat{A} \widehat{\wedge}_{\widehat{\Z}} \widehat{A}$. 
Furthermore $X/ B$ embeds in $ \widehat{X} / \widehat{B} \simeq \varprojlim_i (X_i/B_i)$ by construction since $\cap_i H_j K_{j,i} = H_j$, i.e., $H_j$ is closed in the profinite topology of $B$.
Thus to prove the claim it remains to show that $\widetilde{Z} /B$ embeds in $\varprojlim_i ( \widetilde{Z}_{i}/B_i)$.

 Define
$$Y_1 = \{ H_{j_1} b H_{j_2} \ | \ b \in B, j_1, j_2 \in J, j_1 \not= j_2 \}$$
and
$$ Y_2 = \{  H_j b H_j   \ | \ b \in B \setminus H_j, j \in J  \}$$
 Then $Z / B \simeq Y_1 \cup Y_2$.  This holds because for the element $(H_{j_1} b_1, H_{j_2} b_2) \in X \times X$ its $B$-orbit can be identified with the double coset class $H_{j_1} b_1 b_2^{-1} H_{j_2}$. In this notation $\epsilon(H_{j_1} b H_{j_2} ) =  H_{j_2} b^{-1} H_{j_1}  $. 
Hence $$\widetilde{Z} / B \simeq (Y_1/ \sim_0) \cup (Y_2 / \sim_0)$$ where for  $j_1, j_2 \in J$ we have  $H_{j_1} b_1 H_{j_2} \sim_0 H_{j_1} b_1 H_{j_2}$  and  $H_{j_1} b_1 H_{j_2} \sim_0 H_{j_2} b_1^{-1} H_{j_1}$.
 
 Fix a linear order $(J, <)$. Then 
 $$
 Y_1/ \sim_0 \simeq  \{  H_{j_1} b H_{j_2} \ | \ b \in B, j_1, j_2 \in J, j_1 < j_2 \} = : \widetilde{Y}_1$$
 and 
 $$  Y_2/ \sim_0 \simeq  \{  H_{j} b H_{j} \cup H_j b^{-1} H_j \ | \ b \in B \setminus H_j, j \in J \} = : \widetilde{Y}_2$$
 
 The fact that the map $\widetilde{Y}_1 \cup \widetilde{Y}_2 \simeq \widetilde{Z} /B \to \varprojlim_i ( \widetilde{Z}_{i}/B_i)$ is injective follows from the fact that  if $H, K$ are finitely generated subgroups of a finitely generated free group $F$ then $HgK$ is closed in the profinite topology, see \cite{G-R}, \cite{G-R2}, applied for $F = B$ and $H= H_{j_1}, K= H_{j_2}$.

b) By the universal coefficients theorem there is a short exact sequence
$$0 \to H_2(M, \Z) \otimes_{\Z} (\Z/ q \Z) \to H_2(M, \Z/ q \Z) \to Tor_1^{\Z}(M, \Z/ q \Z) \to 0$$
Since $A$ is a finitely generated, torsion-free group then $M$ is a free $\Z$-module, hence $ Tor_1^{\Z}(M, \Z/ q \Z) = 0$. Hence
$$H_2(M, \Z/ q \Z) \simeq  H_2(M, \Z) \otimes_{\Z} (\Z/ q \Z) \simeq { (M \wedge_{\Z} M) / q (M \wedge_{\Z} M) \simeq M_0 \wedge_{\Z} M_0}$$
where $M_0 =  M/ q M$. Then we can apply the argument from a)  for the group $M_0 = M/ q M$ instead of $M$, $A_0 = A/ qA$ instead of $A$ and get that the map
$$H_0({B},    {M_0}{ {\wedge}_{\Z} } {M_0}) \to H_0(\widehat{B},    \overline{M_0} { \widehat{\wedge}_{\widehat{\Z}}  }
\overline{M_0}) \hbox{ is injective.}$$
Finally by Lemma \ref{wedge2}
$$H_2(\overline{M}, \Z/ q \Z) \simeq  \overline{M_0} \widehat{\wedge}_{\widehat{\Z}} \overline{M_0}$$
 This completes the proof of Claim 2, b).

Finally to conclude the proof of Theorem \ref{comb-wreath-main},a) we apply Claim 1 a), Claim 2 a), Lemma \ref{discrete-homology} a) and Lemma \ref{profinite-homology} a).
To conclude the proof of Theorem \ref{comb-wreath-main},b)  we apply Claim 1 b), Claim 2 b), Lemma \ref{discrete-homology} b) and Lemma \ref{profinite-homology} b).
\end{proof}

\section{More on the homology of the combinatorial wreath product $A \wr_X B$ } \label{homCom}

  We note that the conditions imposed  in the following result require that  $H_1(H, \Z) = 0 = H_2(H, \Z)$. Examples of such groups $H$ are constructed in \cite{B-R-S}.  
 
\begin{theorem} \label{TT0}
Suppose that for every stabilizer $H= stab_B(x)$ of an element $x \in X$ we have that     $ H_2(H, \z) = 0 = H_1(H, \z) $ and the abelianisation $A^{ab}$ is torsion-free. Then for   $G = A \wr_X B$ we have
$$H_2(G, \Z) \simeq H_0(B, H_2(M, \Z)) \oplus H_2(B, \Z)$$
\end{theorem}

\begin{proof} The same proof as the one in Lemma \ref{L1} works. Only we need to justify why $E_{2,1}^2 = 0$ and $E_{1,1}^2 = 0$. 

Note that for the abelianization $M^{ab}$ of $M$ we have that $$M^{ab} \simeq A^{ab} \otimes_{\mathbb{Z}} \Z X$$ as a $\Z B$-module, where the action of $B$ on the right hand side is given by the action of $B$ on $X$. We note that $$A^{ab} \simeq \varinjlim_i A_i,$$ where $\{ A_i \}$ is the direct system of all finitely generated subgroups of $A^{ab}$. 
 We claim that  $$H_j(B, A_i \otimes_{\mathbb{Z}} \Z X)
\simeq   A_i \otimes_{\mathbb{Z}}  H_j(B, \Z X) \hbox{ for } j \in \{ 1, 2\}.$$
Since
$A_i$ is a finite direct sum of copies of $\Z$ it suffices to prove the claim for $A_i = \Z$, but this case is obvious.  Then  for  $m \in \{ 1, 2 \}$
 $$E_{m,1}^2 = H_m(B, M^{ab}) = H_m(B, A^{ab} \otimes_{\mathbb{Z}} \Z X) \simeq 
H_m(B, (\varinjlim_i A_i) \otimes_{\mathbb{Z}} \Z X)
\simeq$$ $$ H_m(B, \varinjlim_i (A_i \otimes_{\mathbb{Z}} \Z X))
\simeq \varinjlim_i  H_m(B, A_i \otimes_{\mathbb{Z}} \Z X)
\simeq \varinjlim_i  (A_i \otimes_{\mathbb{Z}}  H_m(B, \Z X))
\simeq$$ $$ (\varinjlim_i  A_i) \otimes_{\mathbb{Z}}  H_m(B, \Z X)
= A^{ab} \otimes_{\mathbb{Z}}  H_m(B, \Z X).$$
Since we have a decomposition as disjoint union
$X = \cup_{j \in J} H_j   \backslash  B$
where $H_j$ are some stabilizers of elements of $X$ in  $B$, we get for $m \in \{ 1, 2 \}$
$$H_m(B, \Z X) \simeq \oplus_{j \in J} H_m(B, \Z  \otimes_{ \Z H_j} \Z B) \simeq  \oplus_{j \in J} H_m(H_j,  \Z)  = 0$$
hence
 $$E_{m,1}^2 \simeq A^{ab} \otimes_{\Z} H_m(B, \Z X) = 0$$ 
\end{proof}

 \medskip
 Recall that $C = A^{ab} \otimes_{\Z} A^{ab}$ and $  \mu : C \to C$ is
 the linear map defined by $\mu ( \bar{a}_1 \otimes \bar{a}_2) =  \bar{a}_2 \otimes \bar{a}_1$.
 Let $$Y = (X \times X \setminus diag(X))/ B$$
 where $diag(X) = \{ (x,x) \ | \ x \in X \}$ and $B$ acts diagonally on $diag(X)$,  $B$ acts on $X$ on the right. Thus $Y$  can be identified with  the $B$-orbits of $X \times X \setminus diag(X)$.
 Note that the map $X \times X \setminus diag(X) \to X \times X \setminus diag(X)$ that sends $(x_1, x_2) $ to $(x_2, x_1)$ is invariant under the diagonal $B$-action, hence it induces a map $$\epsilon : Y \to Y$$
 Define the $\Z$-module
 $$Z(A, B, X) = \oplus_{y \in Y} C y$$
 where each $Cy \simeq C$ and
 $I(A, B, X)$
 the $\Z$-submodule of $Z(A, B, X)$ generated by $\{ cy + \mu(c) \epsilon(y) \ | \ c \in C, y \in Y \}$.

\begin{theorem} \label{TT1} Suppose that for every stabilizer $H= stab_B(x)$ of an element $x \in X$ we have that    $ H_2(H, \z) = 0 = H_1(H, \z) $ and  the abelianisation $A^{ab}$ is torsion-free. Then 
 $$H_2(A \wr_X B, \Z) \simeq H_2(B, \Z) \oplus (\oplus_{X/ B} H_2(A, \Z)) \oplus ( Z(A, B, X) / I(A, B, X))$$ 
 where $\oplus_{X/ B} H_2(A, \Z)$ denotes the direct sum of $|X/ B|$ copies of $H_2(A, \Z)$.
\end{theorem}

\begin{proof}
The starting point is  Theorem \ref{TT0}  that gives
$$H_2(G, \Z) \simeq H_0(B, H_2(M, \Z)) \oplus H_2(B, \Z)$$
Note that $$M = \oplus_X A \simeq \oplus_{x \in X} A_x$$ where each $A_x \simeq A$ as a group and we fix a linear (total) order of $X$ as a set (that is not linked with the group action).
 Then
 since $H_2$ commutes with direct limits (over directed sets) and using the K\"unneth formula we get that
 $$H_2(M, \Z) \simeq (\oplus_{x \in X} H_2(A_x, \mathbb{Z})) \oplus ( \oplus_{x_1 < x_2; x_1, x_2 \in X} H_1(A_{x_1}, \Z) \otimes_{\Z} H_1(A_{x_2}, \Z) )$$
 Note that
 $$\oplus_{x \in X} H_2(A_x, \mathbb{Z}) \simeq H_2(A, \Z) \otimes_{\Z} \Z X$$
  as a $\Z B$-module. Hence
  \begin{equation}
  \label{iso1} H_2(M, \Z) \simeq (H_2(A, \Z) \otimes_{\Z} \Z X) \oplus ( \oplus_{x_1 < x_2; x_1, x_2 \in B} H_1(A_{x_1}, \Z) \otimes_{\Z} H_1(A_{x_2}, \Z) ) \end{equation}
 Similar to the proof of Theorem \ref{T1}  there is an isomorphism of $\Z B$-modules
 \begin{equation} \label{iso2} 
 V(A, B, X) : = \oplus_{x_1 < x_2; x_1, x_2 \in B} H_1(A_{x_1}, \Z) \otimes_{\Z} H_1(A_{x_2}, \Z) \simeq\end{equation}
 $$\widetilde{V}(A,B, X)/ J$$ where
 $$\widetilde{V}(A,B, X) : =
 \oplus_{(x_1, x_2) \in X \times X \setminus diag(X)}  (A^{ab} \otimes_{\Z} A^{ab}) (x_1, x_2)
 $$
 where  each $(A^{ab} \otimes A^{ab}) (x_1, x_2)$ is a copy $A^{ab} \otimes A^{ab}$ indexed by $(x_1, x_2)$,  the $B$ action on $\widetilde{V}(A, B, X)$ is the diagonal action of $B$ on $X \times X$ and $J$ is the $\Z$-submodule generated by $ \bar{a}_1 \otimes \bar{a}_2 (x_1, x_2) + \bar{a}_2 \otimes \bar{a}_1 (x_2, x_1)$ for all $(x_1, x_2) \in X \times X \setminus diag(X)$ and $\bar{a}_1, \bar{a}_2 \in A^{ab}$.
 Then there is an exact sequence
 $$H_0(B, J) \to H_0(B, \widetilde{V}(A,B, X)) \to H_0(B, V(A,B, X)) \to 0$$
 Similar to the proof of Theorem \ref{T1}
 \begin{equation} \label{iso3} H_0(B, V(A,B, X)) \simeq Z(A, B, X) / I(A, B, X) \end{equation}

Finally by (\ref{iso1}), (\ref{iso2}) and (\ref{iso3})
$$H_0(B, H_2(M, \Z)) \simeq H_0(B, H_2(A, \Z) \otimes_{\Z} \Z X) \oplus H_0(B, V(A,B, X)) $$ $$\simeq 
(\oplus_{X/ B} H_2(A, \Z)) \oplus ( Z(A, B, X) / I(A, B, X)).$$
\end{proof}

\begin{lemma} \label{identification}
If $X$ is the disjoint union of $\cup_{j \in J} x_jB$ and $(J, <)$ is a linear order, then

a) we can identify $X $  with the disjoint union $\cup_{j \in J} H_j\backslash B$, where $H_j \backslash B = \{ H_j b \ | \ b \in B \}$ and $H_j = stab_B(x_j)$. 

b)   $Y = (X \times X \setminus diag(X) ) / B$ can be identified   with the disjoin union $$(\cup_{i \not= j; i,j \in J} H_i \backslash B / H_j) \cup (\cup_{j \in J} ((H_j \backslash B / H_j) \setminus \{ H_j \})) $$
where $H_i \backslash B / H_j$ is the double coset class $\{ H_i b H_j \ | \ b \in B \}$.

c) Under the above identification $\epsilon(H_i b H_j) = H_j b^{-1} H_i$, where $\epsilon : Y \to Y$ sends $(x_1, x_2) \to (x_2, x_1)$.
\end{lemma}

\begin{proof} a) is obvious;

b) for the element $(H_i b_1, H_j b_2) \in X \times X$ its $B$-orbit can be identified by the double coset class $H_i b_1 b_2^{-1} H_j$.

c) since $\epsilon$ is induced by the map $X \times X \to X \times X$ that sends $(H_i b_1, H_j b_2)$ to $(H_j b_2, H_i b_1)$ the result follows from the identification in b).
\end{proof}

Define
$$Y_1 = \{ y = H_i b H_j \ | \ b \in B, i,j \in J, i < j \}$$
and
$$ Y_2 = \{ y = H_j b H_j   \ | \ b \in  B \setminus H_j, j \in J  \}$$

\begin{lemma} \label{LL1} There is an isomorphism of abelian groups
$$Z(A, B, X) / I(A, B, X) \simeq$$ $$
(\oplus_{{  y \in Y_1} } Cy) \oplus (\oplus_{  \{ y, \epsilon(y) \}  \subseteq Y_2, y \not= \epsilon(y) } Cy ) \oplus (\oplus_{y = \epsilon(y) \in Y_2}  (A^{ab} \square_{\Z} A^{ab} )y  )$$
where each   $Cy \simeq C  = A^{ab} \otimes_{\Z} A^{ab}$ and   $(A^{ab} \square_{\Z} A^{ab} )y \simeq A^{ab} \square_{\Z} A^{ab} $.
\end{lemma}
\begin{proof}
It follows from the definitions of $Z(A, B, X)$ and  $I(A, B, X) $  and the identification in Lemma \ref{identification}. Note that when we take direct sum over 
$ y \in Y_1$ we take exactly one $Cy$ for every  $ y \in Y_1$.  Similarly for direct sum over $\{ y, \epsilon(y) \}  \subseteq Y_2, y \not= \epsilon(y)$. Finally when $y = \epsilon(y) \in Y_2$ we have $Cy / \Delta y$, where $\Delta$ is the $\Z$-submodule generated by $c + \mu(c)$, thus $C/ \Delta \simeq  A^{ab} \square_{\Z} A^{ab} $. 
\end{proof}

\section{Grigorchuk group}

Consider the Lysenok presentation \cite{Ly}  of the Grigorchuk group \cite{Grig} :
$$G = \langle a,b,c,d  \ | \ a^2, b^2, c^2, d^2, bcd, \sigma^i( (ad)^4), \sigma^i( (adacac)^4), i \geq 0 \rangle$$
where $\sigma(a) = aca, \sigma(b) = d, \sigma(c) = b, \sigma(d) = c$.

Note that the above presentation contains infinitely many relations, actually by \cite{Grig2} $G$ is not finitely presented. By the above presentation the abelianisation of $G$ is $$ G^{ab} \simeq (\mathbb{Z} / 2 \z) \oplus (\mathbb{Z} / 2 \z) \oplus (\mathbb{Z} / 2 \z)$$ generated by the images of $a,b,c$.

In Benli's paper \cite{Benli} finite quotients $G_n$ of $G$ are studied  for $n \geq 3$. Let
$F$ be the free group  with a free basis $a,b,c,d$ and $K_n$  the normal closure in $F$ of the set $$\{ a^2, b^2, c^2, d^2, bcd, u_0, \ldots, u_{n-3}, v_0, \ldots, v_{n-4}, w_n, t_n \}$$
where $$u_i = \sigma^i ( (ad)^4), v_i = \sigma^i( (adacac)^4), w_n = \sigma^{n-3} ( (ac)^4), t_n = \sigma^{n-3} ( (abac)^4).$$  By definition $$G_n = F/ K_n.$$

\begin{lemma} (\cite{Benli}, Lemma 3.6, Lemma 3.9, Thm 2.6) \label{Ben}  There is an isomorphism
$$K_n/ [K_n, F] \simeq \mathbb{Z}^4 \oplus M_n$$ where

1) $\mathbb{Z}^4$ is generated by $\{ B_1 = a^2, B_2 = b^2, B_3 = c^2, B_4 = bcd \}$, 

2) $M_n$ is  the torsion part of $K_n/ [K_n, F]$ and is   generated by $$Z_n = \{L = b^2 c^2 d^2 (bcd)^{-2}, W_n = \sigma^{n-3}((ac)^4 a^{-4}c^{-4}), T_n = \sigma^{n-3} ((abac)^4a^{-8}b^{-4}c^{-4}) \} $$ $$\cup \{  U_i = \sigma^i( (ad)^4 a^{-4} d^{-4}),  V_j = \sigma^j( ( adacac)^4 a^{-12} d^{-4} c^{-8})\}_{ 0 \leq i \leq n-3, 0 \leq j \leq n-4}$$ 

3) $$M_n \simeq (K_n \cap F') / [K_n, F] \simeq H_2(G_n, \Z) \simeq (\Z/ 2 \Z)^{ 2n - 2 }$$ and $$Z_n\hbox{ is a basis of }H_2(G_n, \Z) \hbox{ as a vector space over }\Z/ 2 \Z.$$
\end{lemma} 
Remark. In \cite{Benli}  there is a misprint as $U_i$ is defined as $\sigma^i( (ad)^4 a^{-4} c^{-4})$ but in the concrete calculations (as in the middle of the proof of \cite[Lemma 3.7]{Benli}) it is used correctly as we stated it above i.e. $U_0 = (ad)^4 a^{-4} d^{-4}$. 

Note that a modification of Lysenok presentations is
$$G = \langle a,b,c,d \ | \ a^2, b^2, c^2, d^2, bcd, \sigma^i( (ad)^4 a^{-4} d^{-4}), \sigma^i( (adacac)^4 a^{-12} d^{-4} c^{-8}), i \geq 0 \rangle$$
Then $G \simeq F/ R$, where $R$ is the normal closure of $$E := \{ a^2, b^2, c^2, d^2, bcd, \sigma^i( (ad)^4 a^{-4} d^{-4}), \sigma^i( (adacac)^4 a^{-12} d^{-4} c^{-8}) \}_{ i \geq 0}$$ 
Let $$\pi : F \to F^{ab} = F/ F' \simeq \Z a \oplus \Z b \oplus \Z c \oplus \Z d$$ be the canonical projection. Note that  $\sigma^i( (ad)^4 a^{-4} d^{-4}), \sigma^i( (adacac)^4 a^{-12} d^{-4} c^{-8}) \in Ker (\pi) = F'$ and   $\pi(E)   = \{ 2a, 2b, 2c, 2d, b+ c + d \}$, here $2(b + c + d)$ is $\Z$-linearly dependent on $ \{ 2a, 2b, 2c, 2d \}$. Hence
$(R \cap F')/ [R, F]$ is an abelian group generated by 
\begin{equation} \label{defZ}  Z = \{  b^2 c^2 d^2 (bcd)^{-2}, \sigma^i( (ad)^4 a^{-4} d^{-4}), \sigma^i( (adacac)^4 a^{-12} d^{-4} c^{-8}) \}_{ i \geq 0} \}\end{equation}

\begin{lemma}\label{basis}  $Z$ is a basis of $H_2(G, \Z)$ as a vector space over $\Z/ 2 \Z$.
\end{lemma}  

\begin{proof}
By the original Grigorchuk paper \cite{Grig2} $$H_2(G, \Z) \simeq (\Z/ 2 \Z)^{\infty}.$$ 
 By the remark 5.10 from Bartholdi-Siegenthaler paper \cite{B1} we have that $H_2(G, \Z)$ is a $R_0$-module for $R_0 = (\Z/ 2 \Z) [\sigma] $ and $$ H_2(G, \Z) \simeq  (\Z/ 2 \Z) \oplus R_0 \oplus R_0$$
as $R_0$-module, where $\Z/ 2 \Z$ is the trivial $R_0$-module i.e. $\Z/ 2 \Z \simeq ((\Z/ 2 \Z) [\sigma]/ ( \sigma - 1))$. 

\medskip
{\bf   Claim} $ \sigma(b^2 c^2 d^2 (bcd)^{-2}) \in  b^2 c^2 d^2 (bcd)^{-2} [R, F]$

\medskip
{\bf  Proof} Note that 
$\sigma(b^2 c^2 d^2 (bcd)^{-2}) = d^2 b^2 c^2 (dbc)^{-2},$ hence
$$(b^2 c^2 d^2 (bcd)^{-2})(d^2 b^2 c^2 (dbc)^{-2})^{-1} =$$ $$ (b^2 c^2 d^2) (bcd)^{-2} (dbc)^2(d^2 b^2 c^2)^{-1} \equiv(bcd)^{-2} (dbc)^2 \hbox{ mod } [R,F]
$$
where above we used $b^2, c^2, d^2 \in R$, so  $ b^2 f \equiv f b^2, c^2 f \equiv f c^2, a^2 f \equiv a^2 f $ mod $[R, F]$ for every $f \in F$. Finally since $bcd \in R$ we have  $(bcd)^{-1} d = d (bcd)^{-1} $ mod $ [R,F]$, hence
$$
(bcd)^{-2} (dbc)^2 =  (bcd)^{-1} (bcd)^{-1} dbc dbc \equiv (bcd)^{-1} d (bcd)^{-1} bc dbc
=$$ $$ (bcd)^{-1} d d^{-1} c^{-1} b^{-1} bc dbc = (bcd)^{-1} dbc \equiv d (bcd)^{-1} bc = 1 \hbox{ mod } [R,F]$$
This completes the proof of the claim.

\medskip
Let $x, y, z$ be the images of $b^2 c^2 d^2 (bcd)^{-2},  (ad)^4 a^{-4} d^{-4}$ and $(adacac)^4 a^{-12} d^{-4} c^{-8}$ in $H_2(G, \Z) \simeq (R \cap F')/ [R, F]$. Write $\overline{Z}$ for the image of $Z$ in $H_2(G, \Z)$.

Note that $R_0$ is a principal ideal domain, hence every finitely generated $R_0$-module is a direct sum of cyclic $R_0$-modules and this decomposition is unique. The generating set $Z$ is $\sigma$-invariant modulo $[R, F]$. Note that by the Claim  $\sigma(x) = x$ and 
$$H_2(G, \Z)  = \langle \overline{Z} \rangle  = R_0 x + R_0 y + R_0 z$$
where $R_0 x \simeq \Z/ 2 \Z$.
On the other hand by the $R_0$-module structure of $H_2(G, \Z)$ as  $   (\Z/ 2 \Z) \oplus R_0 \oplus R_0$ we conclude that $Z$ is a basis of $H_2(G, \Z)$ as a vector space over $\Z/ 2 \Z$.
 \end{proof}

 Recall that for a prime integer $p$ a group $G$  is residually $p$-finite if $\cap_i N_i = 1$ where $G/ N_i$ are finite $p$-groups, i.e., the canonical map $G \to \widehat{G}_p$ is injective, where $\widehat{G}_p$ is the pro-$p$ completion of $G$.

\begin{cor} \label{GG1} Let  $G$ be the Grigorchuk group. Then
 the map $\theta: H_2(G, \Z) \to \varprojlim_n H_2(G_n, {\Z})$  is injective.
In particular $G \wedge G$ is residually  2-finite.
\end{cor}

\begin{proof}
Note that by  the definition of the set $Z$ in (\ref{defZ})  we have that for a finite subset $Z_0$ of $Z$ there is sufficiently big $n$ such that $Z_0$ is a subset of $Z_n$ defined in Lemma \ref{Ben}. Then by Lemma \ref{basis} and  since $Z_n$ is a basis of $H_2(G_n, \Z)$ we conclude  that the map  $\theta$ is injective.  

 We observe that each $G_n$ is a 2-group,  $\varprojlim_n G_n$ is both the profinite and pro-2 completion of $G$. Since $G$ is a self-similar group, it is residually finite. This combined with the fact $G$ is a $2$-group implies it is residually $2$-finite  and $ \widehat{G} \simeq \widehat{G}_2$.   Finally we can apply Theorem \ref{A-L-Dc}, 1) and deduce that $G \wedge G$ is residually  2-finite.
\end{proof}

 \begin{prop} \label{GG20}
 Let  $G$ be the Grigorchuk group. Then $G \otimes G \simeq \Delta(G) \times (G \wedge G)$, in particular $G \otimes G$ is residually  2-finite.
\end{prop}

\begin{proof}
By the Observation after the proof of \cite[Thm. 1.3]{B-F-M} if the abelianization $G^{ab}$ is finitely generated, $G^{ab} = \oplus_{1 \leq i \leq m} \langle x_i G' \rangle
$ with $| x_i | = | x_i G'|$ for all $ 1 \leq i \leq m$ then $\Delta(G) = \Delta(G^{ab})$ and $G \otimes G \simeq \Delta(G) \times (G \wedge G)$.

In our case $G = \langle a, b,c \rangle$ since $ bcd = 1$ and $d$ is redundant as a generator. Furthermore $|a| = |b| = |c| = 2 = |a G'| = | b G'| = | c G' |$, so the above paragraph applies. Since $\Delta(G)$ is a quotient of the Whitehead quadratic functor $\Gamma(G^{ab})$,  $G^{ab}$ is an abelian  finite 2-group and hence $\Gamma(G^{ab})$ is an abelian  finite 2-group,
we conclude that $\Delta(G)$ is an abelian $2$-finite group, hence residually 2-finite. This combined with the conclusion of  Corollary \ref{GG1} that $G \wedge G$ is residually  $2$-finite implies that $G \otimes G$ is residually $2$-finite.
\end{proof} 
 \begin{prop} \label{GG30} Let  $G$ be the Grigorchuk group and $q\geq 2$ be an integer. Then  $G \wedge^q G$ is residually $2$-finite. \end{prop}

\begin{proof}
Let $q = 2^s m$ with $m$ odd. Then $\Z/ q \Z \simeq ( \Z/ 2^s \Z) \oplus (\Z/ m \Z)$ and there are isomorphisms 
$$H_2(G, \Z/ q \Z) \simeq H_2(G,( \Z/ 2^s \Z) \oplus (\Z/ m \Z)) \simeq   H_2(G, \Z/ 2^s \Z) \oplus H_2(G, \Z/ m \Z)$$
and 
$$H_2(\widehat{G}, \Z/ q \Z) \simeq H_2(\widehat{G},( \Z/ 2^s \Z) \oplus (\Z/ m \Z) \simeq   H_2(\widehat{G}, \Z/ 2^s \Z) \oplus H_2( \widehat{G}, \Z/ m \Z)$$

Since $2 H_2(G, \z) = 0$, $G^{ab}$ is a 2-group and $ \widehat{G}$ is a pro-2 group we have that  $H_2(G, \Z/ m \Z) =0 = H_2( \widehat{G}, \Z/ m \Z)$. 

\medskip
{\bf   Claim} The map 
$\rho: H_2(G, \z/ 2^s \Z) \to H_2(\widehat{G}, \Z / 2^s \Z) \hbox{
is injective.}$

{\bf   Proof}
By the universal coefficients theorem  there is an exact sequence
$$0 \to H_2(G, \z) \otimes (\z/ 2^s \z) \to H_2(G, \z/ 2^s \Z) \to Tor_1^{\z}(G^{ab}, \Z/ 2^s \Z) \to 0$$
Since $G^{ab} \simeq (\z/ 2 \Z)^3$ we get  $Tor_1^{\z}(G^{ab}, \Z/ 2^s \Z) \simeq G^{ab}$ and since $2 H_2(G, \Z) = 0$ the above exact sequence can be rewritten as 
\begin{equation} \label{eqeqeq} 0 \to H_2(G, \z)  \to H_2(G, \z/ 2^s \Z) \to G^{ab} \to 0\end{equation}
If we prove that there is an exact sequence
$$0 \to  \varprojlim_n H_2(G_n, \Z) \to H_2(\widehat{G}, \z/ 2^s \Z) \to \widehat{G}^{ab} $$
then since the map $ H_2(G, \z) \to \varprojlim_n H_2(G_n, \Z)$ is injective  and $G^{ab} \to \widehat{G}^{ab}$ is an isomorphism, we conclude that $\rho$ is injective and the last exact sequence is actually a short exact sequence.

We can apply the universal coefficients theorem for $G_n$ and using that $2 H_2(G_n, \Z) = 0$ by Lemma \ref{Ben}, $G_n^{ab} \simeq G^{ab} \simeq (\Z/ 2 \Z)^3$ so
there is an exact sequence
$$0 \to H_2(G_n, \z)  \to H_2(G_n, \z/ 2^s \Z) \to {G}^{ab} \to 0$$
Then take inverse limit and using that inverse limit is a left exact functor we get an exact sequence
$0 \to  \varprojlim_n H_2(G_n, \Z) \to H_2(\widehat{G}, \z/ 2^s \Z) \to \widehat{G}^{ab}  
$. This completes the proof of the Claim.

\medskip
By the Claim and the paragraph before it  the map $H_2(G, \z/ q \Z) \to H_2(\widehat{G}, \Z / q \Z)$
is injective  and 
 by Proposition \ref{iff25} and the fact that $ \widehat{G}$ is the pro-$2$ completion of $G$  we obtain that  $G \wedge^q G$ is residually $2$-finite.

\end{proof}

 \section{Non-goodness} \label{section-nongood}

In \cite{Serre}  Serre 	 called  a discrete group $G$ good (we say (cohomologically) good)  if for the profinite completion $\widehat{G}$ of $G$ the canonical map $G \to \widehat{G}$ induces an isomorphism $$H^i(\widehat{G}, D) \to H^i(G, D)$$ for every finite $G$-module $D$ and every $i \geq 0$, where $H^i(\widehat{G}, D)$ denotes the continuous ( profinite)  cohomology.

By definition $G  \in \mathcal{A}_n$ if the above holds for all $i \leq n$.  By definition   $G  \in \mathcal{A}_n^{hom}$ if  the canonical map $G \to \widehat{G}$ induces an isomorphism $$H_i(G, D) \to H_i(\widehat{G}, D) $$ for every finite ${G}$-module $D$ and  $0 \leq i \leq n$.  
In \cite{BR-K} the authors  proved that homological and cohomological goodness are equivalent i.e   $G \in \mathcal{A}_n$ if and only if $\mathcal{A}_n^{hom}$.

As stated in the introduction the Grigorchuk group is not in $\mathcal{A}_2$. In the following result we show that many residually finite wreath products are not in $\mathcal{A}_2$.

Denote $\mathbb{F}_p = \Z/ p \Z$.

\begin{theorem}(= Theorem D)  Suppose that $B$ is an infinite,  residually finite group, $G = A \wr B$ is a countable, residually finite group  and  $A$  is   a non-trivial  group.  Then   at least one  of the following conditions holds:

a) $A \wr B \not\in {\mathcal{A}}_2$,

b)  $dim_{\mathbb{F}_p}  A  \otimes_{\Z} \mathbb{F}_p \leq 1$ for every prime $p> 2$, $A = 2A$ and $B$ has only finitely many elements of  order different from two.
\end{theorem} 

\begin{proof} 
Note that by Gruenberg criterion since $G$ is residually finite we have that $A$ is abelian and residually finite.

Let $p$ be a prime such that  $A/ p A \not= 0$  i.e.  $dim_{\mathbb{F}_p} (A / pA) \geq 1$.
By Lemma \ref{LL4} $\widehat{G} \simeq \varprojlim_i A_i \wr B_i$ with $\widehat{B} = \varprojlim_i B_i$ and  $\widehat{A} = \varprojlim_i A_i$ with all maps in the inverse limits surjective ones. The canonical map $Z(A_{i+1}, B_{i+1}) \to Z(A_i, B_i)$ is surjective, hence  for  $V_{i} = Z(A_{i}, B_{i}) / I(A_{i}, B_{i})$ the induced map  $V_{i+1} \otimes_{\Z} \mathbb{F}_p \to V_i \otimes_{\Z} \mathbb{F}_p  $ is surjective,  where $Z(-, -)$ and $I(-, -)$ were defined just before Theorem \ref{T1}. By Theorem \ref{T1}
$$H_2(A_i \wr B_i, \Z) \simeq H_2(B_i, \Z) \oplus H_2(A_i, \Z) \oplus ( Z(A_i, B_i) / I(A_i, B_i))$$
By  the universal coefficients theorem $H_2(  A_i \wr B_i ,  \Z)  \otimes_{\Z}  \mathbb{F}_p $ embeds in  $H_2(  A_i \wr B_i , \mathbb{F}_p) $ 
and $H_2(  A_i \wr B_i ,  \Z)   \otimes_{\Z}  \mathbb{F}_p$ contains a copy of $V_i  \otimes_{\Z}  \mathbb{F}_p$. Hence $\varprojlim_iV_i  \otimes_{\Z}  \mathbb{F}_p$ embeds in $ \varprojlim_i H_2(  A_i \wr B_i , \mathbb{F}_p) \simeq H_2( \widehat{G}, \mathbb{F}_p)$.
  Note that  by (\ref{final-eq})
 $$V_i    \otimes_{\Z}  \mathbb{F}_p \simeq  \oplus_{ \{b, b^{-1} \} \subseteq B_i \setminus \{ 1 \},  |b| \not= 2 }  ( (A_i    \otimes_{\Z}  A_i)    \otimes_{\Z}  \mathbb{F}_p )
 \oplus ( \oplus_{ b \in B_i \setminus \{ 1 \}, |b| = 2 }  (A_i    \square_{\Z}  A_i)    \otimes_{\Z}  \mathbb{F}_p )$$ 
 Then
  for $dim = dim_{\mathbb{F}_p}$ $$dim(V_i    \otimes_{\Z}  \mathbb{F}_p) \geq dim( (A_i    \otimes_{\Z}  A_i)    \otimes_{\Z}  \mathbb{F}_p ) | \{  \{b, b^{-1} \} \subseteq  B_i \setminus \{ 1 \} \ | \ |b| \not= 2 \} | + $$ $$dim( (A_i  \square_{\Z} A_i)    \otimes_{\Z} \mathbb{F}_p ) | \{ b \in B_i \setminus \{ 1 \} \ | \ |b| = 2 \} |$$
 If $dim (A/ p A) \geq 2$  if $p>2$ or $dim (A/ 2 A) \geq 1$ if $p = 2$  then
 $$dim(V_i    \otimes_{\Z}  \mathbb{F}_p) \geq | \{   \{b, b^{-1} \} \subseteq   B_i \setminus \{ 1 \} \ | \ |b| \not= 2 \} | +  | \{ b \in B_i \setminus \{ 1 \} \ | \ |b| = 2 \} | \geq \frac{ | B_i | - 1}{2}$$
 Since $B$ is infinite and $B =  \varprojlim_i B_i$ we conclude that $\lim_{i\to \infty} | B_i | = \infty$. Hence $\varprojlim_i V_i    \otimes_{\Z}  \mathbb{F}_p$ is not countable, and so $H_2( \widehat{G}, \mathbb{F}_p)$ is not countable.
 
 On the other hand  $H_2(G, \mathbb{F}_p)$ is countable. Indeed if $G \simeq F/ R$ is a countable group,  where $F$ is a countable free group   and  $H_2(G, \mathbb{F}_p) \simeq R \cap ([F,F]F^p) / [F, R] R^p$ is countable since it is a section of $F$ and $F$ is countable. Then the map $H_2(G, \mathbb{F}_p) \to H_2 ( \widehat{G}, \mathbb{F}_p)$ is not an isomorphism, hence $G \not\in \mathcal{A}_2^{hom} = \mathcal{A}_2$.

It remains to consider the case when  $dim (A/ p A) \leq 1$ for all $p >2$ and $A = 2A$. Since $A$ is residually finite and abelian  $\cap_{n \geq 2}  n A$ is the trivial group and since $A$ is non-trivial there is some $p > 2$ such that $A/ p A $ is non-trivial, hence  $dim (A/ p A) = 1$. 
 Then $dim( (A_i    \otimes_{\Z}  A_i)    \otimes_{\Z}  \mathbb{F}_p )  = 1$, hence
$$dim(V_i    \otimes_{\Z}  \mathbb{F}_p) \geq | \{    \{b, b^{-1} \} \subseteq   B_i \setminus \{ 1 \} \ | \ |b| \not= 2 \} |$$
If $\lim_i | \{   \{b, b^{-1} \} \subseteq   B_i \setminus \{ 1 \} \ | \ |b| \not= 2 \} | = \infty$ we finish as in the previous case. Suppose this limit is $k$.

Since $\widehat{B} = \varprojlim_i B_i$ where all maps in the inverse limit are surjective and  the canonical map $B \to \widehat{B}$ is injective, we get that  for $b \in B \setminus \{ 1 \}$ if $|b| \not= 2$ then there is some $i$ such that the image of $b$ in $B_i$ has order different from 2. Hence $B$ is an infinite group  where at most  $2k$ elements have  order different from two.
\end{proof}

\end{document}